\documentclass[a4paper,11pt]{article}
\usepackage{amsmath,amsthm,amssymb}
\usepackage[mathscr]{eucal}
\usepackage[bookmarks=false]{hyperref}

\numberwithin{equation}{section}

{\theoremstyle{definition}\newtheorem{definition}{Definition}[section]

}

\newtheorem{proposition}[definition]{Proposition}
\newtheorem{lemma}[definition]{Lemma}
\newtheorem{theorem}[definition]{Theorem}

\newcommand{\M}{\operatorname{M}}
\newcommand{\C}{\mathbb{C}}
\newcommand{\F}{\mathbb{F}}
\newcommand{\cR}{\mathcal{R}}
\newcommand{\actson}{\curvearrowright}
\newcommand{\SL}{\operatorname{SL}}
\newcommand{\rL}{\mathord{\text{\rm L}}}

\newcommand{\N}{\mathbb{N}}

\newcommand{\Z}{\mathbb{Z}}
\newcommand{\cF}{\mathcal{F}}
\newcommand{\cA}{\mathcal{A}}
\newcommand{\cV}{\mathcal{V}}
\newcommand{\id}{\mathord{\operatorname{id}}}

\newcommand{\recht}{\rightarrow}
\newcommand{\cU}{\mathcal{U}}
\newcommand{\vphi}{\varphi}
\newcommand{\cW}{\mathcal{W}}
\newcommand{\R}{\mathbb{R}}

\newcommand{\eps}{\varepsilon}

\newcommand{\ovt}{\mathbin{\overline{\otimes}}}
\newcommand{\B}{\operatorname{B}}
\newcommand{\om}{\omega}
\newcommand{\cP}{\mathcal{P}}
\newcommand{\cZ}{\mathcal{Z}}
\newcommand{\Q}{\mathbb{Q}}

\newcommand{\cK}{\mathcal{K}}

\newcommand{\cH}{\mathcal{H}}
\newcommand{\cJ}{\mathcal{J}}
\newcommand{\ot}{\otimes}

\newcommand{\dis}{\displaystyle}
\newcommand{\Ad}{\operatorname{Ad}}
\newcommand{\cG}{\mathcal{G}}
\newcommand{\cM}{\mathcal{M}}
\newcommand{\dpr}{^{\prime\prime}}

\newcommand{\vphitil}{\widetilde{\vphi}}

\newcommand{\lspan}{\operatorname{span}}

\newcommand{\cN}{\mathcal{N}}

\newcommand{\cS}{\mathcal{S}}

\newcommand{\Om}{\Omega}

\newcommand{\cC}{\mathcal{C}}
\newcommand{\qtil}{\widetilde{q}}
\newcommand{\cQ}{\mathcal{Q}}

\newcommand{\cO}{\mathcal{O}}

\newcommand{\otalg}{\otimes_{\text{\rm alg}}}

\newcommand{\ubar}{\overline{u}}

\newcommand{\abar}{\overline{a}}

\newcommand{\zbar}{\overline{z}}
\newcommand{\Gammah}{\widehat{\Gamma}}
\newcommand{\Ktil}{\widetilde{K}}

\newcommand{\cT}{\mathcal{T}}
\newcommand{\Cstarred}{\operatorname{C}^*_{\text{\rm red}}}
\newcommand{\Omtil}{\widetilde{\Omega}}

\newcommand{\m}{\mathord{\text{\rm m}}}
\newcommand{\SO}{\operatorname{SO}}
\newcommand{\SU}{\operatorname{SU}}
\newcommand{\Prob}{\operatorname{Prob}}
\newcommand{\cD}{\mathcal{D}}

\newcommand{\op}{^\text{\rm op}}
\newcommand{\cb}{_\text{\rm cb}}

\newcommand{\QHreg}{\mathcal{QH}_{\text{\rm\tiny reg}}}
\newcommand{\etatil}{\widetilde{\eta}}
\newcommand{\Gammatil}{\widetilde{\Gamma}}
\newcommand{\Lambdah}{\widehat{\Lambda}}
\newcommand{\rLUC}{\operatorname{LUC}}

\newcommand{\cCs}{\mathcal{C}_{\text{\rm\tiny s}}}

\begin{document}

\begin{center}
{\boldmath\LARGE\bf Unique Cartan decomposition for II$_1$ factors\vspace{0.5ex}\\
arising from arbitrary actions of hyperbolic groups}

\bigskip

{\sc by Sorin Popa\footnote{Mathematics Department, UCLA, CA 90095-1555 (United States), popa@math.ucla.edu\\
Supported in part by NSF Grant DMS-1101718} and Stefaan Vaes\footnote{K.U.Leuven, Department of Mathematics, Leuven (Belgium), stefaan.vaes@wis.kuleuven.be \\
    Supported by ERC Starting Grant VNALG-200749, Research
    Programme G.0639.11 of the Research Foundation --
    Flanders (FWO) and K.U.Leuven BOF research grant OT/08/032.}}
\end{center}

\begin{abstract}\noindent
We prove that for any free ergodic probability measure preserving action $\Gamma \actson (X,\mu)$ of a non-elementary hyperbolic group, or a lattice in a rank one simple Lie group, the associated group measure space II$_1$ factor $\rL^\infty(X) \rtimes \Gamma$ has $\rL^\infty(X)$ as its unique Cartan subalgebra, up to unitary conjugacy.
\end{abstract}

\section{Introduction and main results}

A {\it Cartan subalgebra} $A$ in a (separable)
II$_1$ factor $M$ is a maximal abelian $*$-subalgebra $A\subset M$
with normalizer $\cN_M(A)=\{u\in \cU(A) \mid
uAu^*=A\}$ generating $M$.
Its presence amounts to realizing $M$ as a generalized (twisted)
version of the group measure space construction, for a measure preserving ergodic countable equivalence relation $\cR$ on a probability
space $X$ and a 2-cocycle $v$ for $\cR$. Showing uniqueness (up to conjugacy by an automorphism)
of Cartan subalgebras  is important, because the
classification of factors $M$ satisfying this property reduces to the classification of the
associated pairs $(\cR, v)$ (\cite{FM75}). In particular, the classification of group measure space factors $M=\rL^\infty(X)\rtimes \Gamma$
with unique Cartan subalgebras, reduces to the classification up to orbit equivalence of the corresponding
free ergodic probability measure preserving (pmp) actions $\Gamma \curvearrowright X$.

It has been known since \cite{CFW81} that
any two Cartan subalgebras of the hyperfinite II$_1$ factor $R$ are conjugated by an automorphism, and thus
any 2-cocycle of any free ergodic pmp action
of an amenable group vanishes (untwists) and any two ergodic actions of any
two amenable groups are orbit equivalent. While in the the nonamenable case examples of
group measure space factors with two distinct Cartan subalgebras
were already constructed in \cite{CJ81}, uniqueness results started to emerge in \cite{Po01}, where it was shown
that all Cartan subalgebras $A\subset M$ that satisfy a certain \emph{rigidity} property in a factor of the form  $\rL^\infty(X)\rtimes \F_n$,
with $\F_n$ being the free group on $2\leq n \leq \infty$ generators, is unitarily conjugate to $\rL^\infty(X)$.
This led to the conjecture that such a property could hold without any condition on the Cartan subalgebra.
Further supporting evidence came with the work in \cite{OP07}, where it was shown that group measure space factors
arising from \emph{profinite} actions of $\F_n$ have unique Cartan decomposition.

We solved this conjecture in \cite{PV11}, where we actually found a large class of groups $\Gamma$,
containing $\F_n$, with the property that the II$_1$ factor $\rL^\infty(X) \rtimes \Gamma$ associated with an \emph{arbitrary} free ergodic pmp action $\Gamma \actson (X,\mu)$ has $\rL^\infty(X)$ as its unique Cartan subalgebra up to unitary conjugacy, i.e.\ $\Gamma$ is \emph{$\cC$-rigid}, in the sense of \cite[Definition 1.4]{PV11}. More precisely, we
showed in \cite[Theorem 1.2]{PV11} that all weakly amenable groups that admit a proper $1$-cocycle into a nonamenable representation are $\cC$-rigid.
To prove this result, we first showed in \cite[Theorem 5.1]{PV11}
(by only using the weak amenability of $\Gamma$!)  that the normalizer of any Cartan subalgebra $A \subset \rL^\infty(X) \rtimes \Gamma$ has a special almost invariance property, that can be viewed as a relative version (w.r.t.\ $\rL^\infty(X)$) of the notion of weak compactness in \cite[Definition 3.1]{OP07}. The second part of the proof consisted in applying to this relative weak compactness the malleable deformation
associated in \cite{Si10} with a $1$-cocycle into an orthogonal representation of $\Gamma$. As such, we derived that if $A$ is not unitarily conjugate to $\rL^\infty(X)$ then its normalizer generates an amenable subalgebra (thus contradicting the regularity of $A$).

The degree of generality of the results in \cite{PV11} was thus limited by the assumption that
$\Gamma$ admits a proper $1$-cocycle into a nonamenable orthogonal representation $\eta$ on $K_\R$, i.e., of a proper map $c : \Gamma \recht K_\R$ satisfying $c(gh) = c(g) + \eta_g c(h)$ for all $g, h \in \Gamma$.

In the particular case of profinite actions, this type of limitation had already been circumvented
in \cite{CS11} under the weaker assumption that the group $\Gamma$ belongs to their class $\QHreg$, requiring that $\Gamma$ has an orthogonal representation $\eta$ on $K_\R$ that is weakly contained in the left regular representation and that merely admits a
proper map $c : \Gamma \recht K_\R$ coarsely satisfying the $1$-cocycle relation, i.e.\ $\sup_{k \in \Gamma} \|\eta_g c(k) - c(gkh)\| < \infty$, $\forall g,h \in \Gamma$. Thus,   it is shown in \cite{CS11} that for all profinite free ergodic pmp actions of all weakly amenable, nonamenable groups in the class $\QHreg$, the crossed product has a unique Cartan subalgebra up to unitary conjugacy. This result was then extended in \cite{CSU11} to cover as well products of weakly amenable groups in $\QHreg$.

As we will later explain, the class of exact groups in $\QHreg$ coincides with the class of bi-exact groups in the sense of \cite{Oz03} (see Definition \ref{def.wa-bi-exact} and Proposition \ref{prop.biexact} below). In this paper, which should be viewed as a follow-up to \cite{PV11}, we show that weakly amenable, nonamenable, bi-exact groups are in fact $\cC$-rigid, i.e., {\it all}  their group measure space
factors have unique Cartan subalgebra. To prove this result, we first use the relative weak compactness property (which was obtained in \cite[Theorem 5.1]{PV11} from the weak amenability assumption) and then apply the bi-exactness property, by using an argument
inspired by the proof of \cite[Theorem 15.1.5]{BO08}. More precisely, we obtain the following general result.

\begin{theorem}\label{thm.main-hyp}
Let $\Gamma$ be a weakly amenable, nonamenable, bi-exact group, or let $\Gamma$ be a direct product of $1 \leq n < \infty$ such groups. If $\Gamma \actson (X,\mu)$ is an arbitrary free ergodic pmp action, then $\rL^\infty(X)$ is the unique Cartan subalgebra of $\rL^\infty(X) \rtimes \Gamma$, up to unitary conjugacy.

In particular, all of the following groups are $\cC$-rigid.
\begin{enumerate}
\item non-elementary hyperbolic groups,
\item lattices in a connected noncompact rank one simple Lie group with finite center,
\item limit groups in the sense of Sela,
\item direct products of $1 \leq n < \infty$ groups as in 1, 2 and 3.
\end{enumerate}
\end{theorem}

One should point out that, although in our proof of Theorem \ref{thm.main-hyp} we use an approach based on bi-exactness rather than the $\QHreg$ property,
we owe much to ideas in \cite{CS11}, on how to go beyond groups admitting proper $1$-cocycles. In fact, in a first version of this paper we
gave a proof of Theorem \ref{thm.main-hyp} using the methods of \cite{CS11}, before we found the present much simpler and direct  argument.

Recall from \cite[Definition 1.4]{PV11} the following definition.

\begin{definition}\label{def.Crigid}
We say that a countable group $\Gamma$ is $\cC$-rigid (Cartan-rigid) if for every free ergodic pmp action $\Gamma \actson (X,\mu)$, the II$_1$ factor $\rL^\infty(X) \rtimes \Gamma$ has $\rL^\infty(X)$ as its unique Cartan subalgebra up to unitary conjugacy.

In view of \cite[Proposition 4.12]{OP07}, we say that a countable group $\Gamma$ is $\cCs$-rigid if for every free ergodic pmp action $\Gamma \actson (X,\mu)$, the II$_1$ factor $M = \rL^\infty(X) \rtimes \Gamma$ has the following property~: every maximal abelian subalgebra $A \subset M$ whose normalizer $\cN_M(A)\dpr$ is a finite index subfactor of $M$, is unitarily conjugate to $\rL^\infty(X)$.
\end{definition}

The groups $\Gamma$ in Theorem \ref{thm.main-hyp} are in fact $\cCs$-rigid. Moreover the same holds for all groups that are measure equivalent with $\Gamma$ (see Definition \ref{def.ME}).

\begin{theorem}\label{thm.main-ME}
The following groups are $\cC$-rigid and $\cCs$-rigid~:
\begin{enumerate}
\item groups that are measure equivalent to a direct product of $1 \leq n < \infty$ weakly amenable, nonamenable, bi-exact groups,
\item countable closed subgroups $\Gamma$ of a direct product $G = G_1 \times \ldots \times G_n$ of connected noncompact rank one simple Lie groups $G_i$ with finite center, such that the image of $\Gamma$ in $G_i$ has a nonamenable closure for all $1 \leq i \leq n$.
\end{enumerate}
\end{theorem}

Following \cite{OP07}, a finite von Neumann algebra $M$ is called \emph{strongly solid} if the normalizer of any diffuse amenable subalgebra of $M$ is still amenable. It is shown in \cite{OP07} that the free group factors $\rL \F_n$ are strongly solid. As explained above, it was proven in \cite{CS11} that in fact the group von Neumann algebras $\rL \Gamma$ of all hyperbolic groups $\Gamma$ are strongly solid.
Crossed products $B \rtimes \Gamma$ are typically not strongly solid, but we establish the following relative strong solidity property: for weakly amenable, bi-exact groups $\Gamma$, we prove the dichotomy that if a subalgebra $A$ of a crossed product $B \rtimes \Gamma$ is amenable relative to $B$, then either $A$ embeds into $B$ (in the sense of intertwining-by-bimodules, see Definition \ref{def.intertwine}), or $A$ has a normalizer that remains amenable relative to $B$.

\begin{theorem}\label{thm.main-two}
Let $\Gamma$ be a weakly amenable, bi-exact group and let $\Gamma \actson (B,\tau)$ be an arbitrary trace preserving action on the tracial von Neumann algebra $(B,\tau)$. Put $M = B \rtimes \Gamma$.

If $q \in M$ is a projection and $A \subset qMq$ is a von Neumann subalgebra that is amenable relative to $B$, then $\cN_{qMq}(A)\dpr$ remains amenable relative to $B$, or $A \prec_M B$.
\end{theorem}

Both Theorem \ref{thm.main-hyp} and \ref{thm.main-two} will be deduced in Section \ref{sec.proofs} from our more technical Theorem \ref{thm.key}, also yielding the following new class of tensor product II$_1$ factors without Cartan subalgebras, improving \cite[Corollary 2]{OP07}.

\begin{theorem}\label{thm.no-Cartan}
Let $\Gamma$ be a nonamenable, icc, weakly amenable, bi-exact group and let $N$ be an arbitrary II$_1$ factor. Then $N \ovt \rL \Gamma$ has no Cartan subalgebra.
\end{theorem}

A statement similar to \ref{thm.main-two} holds for direct product groups and goes as follows. We use the strong intertwining notation $\prec^f$ that is introduced in Definition \ref{def.intertwine} below.

\begin{theorem}\label{thm.products}
Let $\Gamma = \Gamma_1 \times \cdots \times \Gamma_n$ be the direct product of $n \geq 1$ weakly amenable, bi-exact groups $\Gamma_i$. Let $\Gamma \actson (B,\tau)$ be an arbitrary trace preserving action on the tracial von Neumann algebra $(B,\tau)$. Put $M = B \rtimes \Gamma$. Let $A \subset qMq$ be a von Neumann subalgebra that is amenable relative to $B$ and put $P := \cN_{qMq}(A)\dpr$.

Then there exist projections $p_0,\ldots,p_n \in \cZ(P)$, some of which might be zero, such that\linebreak $p_0 \vee \cdots \vee p_n = q$ and
\begin{itemize}
\item $P p_0$ is amenable relative to $B$,
\item for every $i=1,\ldots,n$ we have $A p_i \prec^f_M B \rtimes \Gammah_i$ where $\Gammah_i$ is the direct product of all $\Gamma_j$, $j \neq i$.
\end{itemize}
\end{theorem}

Note that results of the same type as Theorems \ref{thm.main-two}, resp.\ \ref{thm.products}, were established in \cite{CS11}, resp.\ \cite{CSU11}, under the additional assumption that $A \subset qMq$ is a \emph{weakly compact embedding} and that $A$ and $B$ are amenable von Neumann algebras.

Since for $\cC$-rigid groups $\Gamma$, the classification of group measure space factors $\rL^\infty(X) \rtimes \Gamma$ reduces to the classification of the associated free ergodic pmp actions $\Gamma \actson (X,\mu)$ up to orbit equivalence (OE), Theorem \ref{thm.main-hyp} can be combined with existing OE rigidity results, in particular with the work of \cite{MS02} on OE rigidity for direct products of hyperbolic groups. This leads to the following result. We refer to Section \ref{sec.rigidity} for terminology and to \cite[Section 12]{PV11} for further applications in W$^*$-superrigidity.

\begin{theorem}\label{thm.MoSh}
Let  $\Gamma = \Gamma_1 \times \Gamma_2$ be the direct product of two non-elementary hyperbolic groups. Assume that $\Gamma \actson (X,\mu)$ is a free ergodic pmp action that is aperiodic and irreducible.

If $\rL^\infty(X) \rtimes \Gamma \cong \rL^\infty(Y) \rtimes \Lambda$ for any free mildly mixing pmp action $\Lambda \actson (Y,\eta)$, then $\Gamma \cong \Lambda$ and the actions $\Gamma \actson X$ and $\Lambda \actson Y$ are conjugate.
\end{theorem}

\section{Preliminaries}

Throughout this article we call \emph{tracial von Neumann algebra $(M,\tau)$,} any von Neumann algebra $M$ equipped with a faithful normal tracial state $\tau$.

\subsection{Intertwining by bimodules}

We recall from \cite[Theorem 2.1 and Corollary 2.3]{Po03} the theory of \emph{intertwining-by-bimodules,} summarized in the following definition.

\begin{definition}\label{def.intertwine}
Let $(M,\tau)$ be a tracial von Neumann algebra and $P,Q \subset M$ possibly non-unital von Neumann subalgebras. We write $P \prec_M Q$ when one of the following equivalent conditions is satisfied.
\begin{itemize}
\item There exist projections $p \in P$, $q \in Q$, a normal $*$-homomorphism $\vphi : pPp \recht qQq$ and a nonzero partial isometry $v \in pMq$ such that $x v = v \vphi(x)$ for all $x \in pPp$.
\item It is impossible to find a net of unitaries $u_n \in \cU(P)$ satisfying $\|E_Q(x u_n y^*)\|_2 \recht 0$ for all $x,y \in 1_Q M 1_P$.
\end{itemize}
We write $P \prec^f_M Q$ if $P p \prec_M Q$ for every projection $p \in P' \cap 1_P M 1_P$.
\end{definition}

\subsection{Jones' basic construction}

Let $(M,\tau)$ be a tracial von Neumann algebra and $B \subset M$ a von Neumann subalgebra. Jones' basic construction $\langle M,e_B \rangle$ is defined as the von Neumann algebra acting on $\rL^2(M)$ generated by $M$ and the orthogonal projection $e_B$ of $\rL^2(M)$ onto $\rL^2(B)$. Recall that $\langle M,e_B \rangle$ coincides with the commutant of the right $B$-action on $\rL^2(M)$.

\subsection{Relative amenability}

Recall that a functional $\Omega$ on a von Neumann algebra $\cN$ with subalgebra $P \subset \cN$ is called \emph{$P$-central} if $\Om(x S) = \Om(S x)$ for all $x \in P$, $S \in \cN$.

Let $(M,\tau)$ be a tracial von Neumann algebra, $p \in M$ a projection and $P \subset pMp ,B \subset M$ von Neumann subalgebras. Following \cite[Section 2.2]{OP07} we say that $P$ is \emph{amenable relative to $B$} if the von Neumann algebra $p \langle M, e_B \rangle p$ admits a $P$-central positive functional whose restriction to $pMp$ coincides with $\tau$. We need the following variant of \cite[Corollary 2.6]{PV11}.

\begin{lemma}\label{lem.rel-amen}
Let $(M,\tau)$ be a tracial von Neumann algebra. Assume that $P_1 \subset P_2 \subset M$ and $Q \subset M$ are von Neumann subalgebras such that $P_1 \subset P_2$ is a finite index subfactor.

If $p_1 \in P_1' \cap M$ is a nonzero projection such that $P_1 p_1$ is amenable relative to $Q$ and if $p_2$ denotes the smallest projection in $P_2' \cap M$ that dominates $p_1$, then $P_2 p_2$ is amenable relative to $Q$.
\end{lemma}
\begin{proof}
Take a Pimsner-Popa basis (see \cite[Proposition 1.3]{PP84}) for the finite index subfactor $P_1 \subset P_2$~: we find elements $v_1,\ldots,v_n \in P_2$ and a projection $q \in \M_n(\C) \ot P_1$ such that the map $U : q(\C^n \ot \rL^2(P_1)) \recht \rL^2(P_2) : U(q(e_i \ot x)) = v_i x$ for all $i=1,\ldots,n$, $x \in P_1$, is a unitary operator. Define the normal $*$-homomorphism $\vphi : P_2 \recht q(\M_n(\C) \ot P_1)q$ such that $U(\vphi(x) \xi) = x U(\xi)$ for all $x \in P_2$ and $\xi \in q(\C^n \ot \rL^2(P_1))$. Defining $V \in \M_{1,n}(\C) \ot P_2$ given by $V = \sum_{i=1}^n e_{1i} \ot v_i$, we get that $x V = V \vphi(x)$ for all $x \in P_2$.

Write $T := \sum_{i=1}^n v_i p_1 v_i^*$. A direct computation shows that $T$ is a positive element in $P_2' \cap M$. The support projection of $T$ equals the projection onto the closed linear span of $\{v_i p_1 x \mid i=1,\ldots,n, x \in M\}$. Since $p_1$ commutes with $P_1$ and since the linear span of $v_i P_1$ equals $P_2$, it follows that the support projection of $T$ equals the projection onto the closed linear span of $P_2 p_1 M$. Thus, the support projection of $T$ equals $p_2$.

Since $P_1 p_1$ is amenable relative to $Q$, we get a $P_1 p_1$-central positive functional $\Om_1$ on $p_1 \langle M,e_Q\rangle p_1$ such that $\Om_1(x) = \tau(x)$ for all $x \in p_1 M p_1$. Define the positive functional $\Om_2$ on $p_2 \langle M,e_Q \rangle p_2$ given by
$$\Om_2(S) = \sum_{i=1}^n \Om_1(p_1 v_i^* S v_i p_1) \; .$$
A direct computation shows that $\Om_2$ is $P_2 p_2$-central. Also, for all $x \in p_2 M p_2$, we have that $\Om_2(x) = \tau(x T)$. Since $T \in P_2' \cap M$ and since the support projection of $T$ equals $p_2$, we can take a sequence of positive elements $T_n \in P_2' \cap M$ such that $T_n T = T T_n \leq p_2$ and $T_n T \recht p_2$ strongly. If we choose the positive functional $\Om$ on $p_2 \langle M,e_Q \rangle p_2$ as a weak$^*$-limit point of the sequence of positive functionals $S \mapsto \Om_2(T_n^{1/2} S T_n^{1/2})$, it follows that $\Om$ is a $P_2 p_2$-central positive functional on $p_2 \langle M,e_Q \rangle p_2$ with $\Om(x) = \tau(x)$ for all $x \in p_2 M_2 p_2$. Hence, $P_2 p_2$ is amenable relative to $Q$.
\end{proof}

\subsection{\boldmath Bi-exactness and the classes $\QHreg$ and $\cS$}\label{subsec.biexact}

\begin{definition}\label{def.wa-bi-exact}
Let $\Gamma$ be a countable group.
\begin{itemize}
\item (\cite{CH88}) The group $\Gamma$ is called \emph{weakly amenable} if there exists a sequence of finitely supported functions $f_n : \Gamma \recht \C$ tending to $1$ pointwise and satisfying $\sup_n \|f_n\|\cb < \infty$. Here $\|f\|\cb$ is the \emph{Herz-Schur norm}, i.e.\ the cb-norm of the linear map $\rL(\Gamma) \recht \rL(\Gamma) : u_g \mapsto f(g) u_g$.
\item (\cite[Definition 15.1.2]{BO08}) The group $\Gamma$ is called \emph{bi-exact} if $\Gamma$ is exact and if there exists a map $\mu : \Gamma \recht \Prob \Gamma$ satisfying
\begin{equation}\label{eq.bi-exact}
\lim_{k \recht \infty} \|\mu(g k h) - g \cdot \mu(k) \|_1 = 0 \quad\text{for all}\;\; g,h \in \Gamma \; .
\end{equation}
\end{itemize}
\end{definition}

Collecting several results from the literature, we get the following large classes of weakly amenable, bi-exact groups.

\begin{lemma}\label{lem.exam-wa-bi-exact}
The following groups are weakly amenable and bi-exact~:
\begin{itemize}
\item (\cite{Oz03,Oz07}) non-elementary hyperbolic groups,
\item (\cite{CS11}) lattices in a connected noncompact rank one simple Lie group with finite center,
\item (\cite{Oz05,Oz12}) limit groups in the sense of Sela.
\end{itemize}
Finally, by \cite{Sa09,Oz10}, the family of weakly amenable, bi-exact groups is stable under measure equivalence and under the passage to measure equivalence subgroups (see Definition \ref{def.ME}).
\end{lemma}

Before proving Lemma \ref{lem.exam-wa-bi-exact}, recall from \cite[Section 4]{Oz04} that the class $\cS$ is defined as the class of countable groups $\Gamma$ for which the action of $\Gamma \times \Gamma$ by left right multiplication on the Stone-\v{C}ech remainder $\partial^\beta(\Gamma)$ of $\Gamma$ is topologically amenable. By \cite[Proposition 15.2.3]{BO08}, a countable group $\Gamma$ belongs to the class $\cS$ if and only if $\Gamma$ admits a compactification $\Gamma \subset X$ such that
\begin{itemize}
\item the left multiplication action of $\Gamma$ extends to an action of $\Gamma$ by homeomorphisms of $X$ that is topologically amenable;
\item the right multiplication action of $\Gamma$ extends to an action of $\Gamma$ by homeomorphisms of $X$ that are equal to the identity on $X - \Gamma$.
\end{itemize}
By \cite[Proposition 15.2.3]{BO08}, a countable group $\Gamma$ belongs to the class $\cS$ if and only if $\Gamma$ is bi-exact in the sense of Definition \ref{def.wa-bi-exact}.

\begin{definition}\label{def.ME}
A countable group $\Gamma$ is called a \emph{measure equivalence subgroup} of a countable group $\Lambda$ if $\Gamma \times \Lambda$ admits a measure preserving action on a, typically infinite, standard measure space $(\Omega,m)$ such that both the actions $\Gamma \actson \Om$ and $\Lambda \actson \Om$ are free and admit a fundamental domain, with the fundamental domain of $\Lambda \actson \Om$ having finite measure.

If the action $\Gamma \times \Lambda \actson \Om$ can be chosen in such a way that also the fundamental domain of $\Gamma \actson \Om$ has finite measure, the groups $\Gamma$ and $\Lambda$ are called \emph{measure equivalent}.
\end{definition}

\begin{proof}[Proof of Lemma \ref{lem.exam-wa-bi-exact}]
As explained above, by \cite[Proposition 15.2.3]{BO08}, the class $\cS$ equals the class of bi-exact groups.

The action of a word hyperbolic group on its Gromov boundary is topologically amenable (see e.g.\ \cite[Theorem 5.3.15]{BO08}) and hence all hyperbolic groups belong to the class $\cS$. By \cite[Proposition 12]{Oz05} also groups that are hyperbolic relative to a family of amenable subgroups belong to the class $\cS$. Since by \cite[Theorem 0.3]{Da02}, Sela's limit groups are hyperbolic relative to a family of cyclic subgroups, they belong to class $\cS$. By \cite[Theorem 3.1]{Sa09}, the class $\cS$ is stable under the passage to ME-subgroups, and in particular under measure equivalence. Since a lattice $\Gamma$ in a connected noncompact rank one simple Lie group with finite center is measure equivalent with a cocompact lattice $\Lambda$ in the same Lie group, and since such a $\Lambda$ is hyperbolic and therefore belongs to class $\cS$, also $\Gamma$ belongs to class $\cS$.

From \cite{CH88} we know that lattices in connected noncompact rank one simple Lie groups with finite center, are weakly amenable. By \cite{Oz07} hyperbolic groups are weakly amenable. The following argument of \cite{Oz12} shows that Sela's limit groups $\Gamma$ are weakly amenable. The group $\Gamma$ is a subgroup of an ultraproduct of free groups. Since all free groups are a subgroup of $\SL(2,\Z)$, we can view $\Gamma$ as a subgroup of $\SL(2,\Z)^\om$, for some free ultrafilter $\om$ on $\N$. Denoting by $K$ the ultrapower field $K := \Q^\om$, we see that $\Gamma < \SL(2,K)$. In \cite[Theorem 4]{GHW04} it is shown that all countable subgroups of $\SL(2,K)$ have the Haagerup approximation property. The same argument actually shows that they are as well weakly amenable.

Finally, it was proven in \cite[End of Section 2]{Oz10} that weak amenability is stable under the passage to ME-subgroups.
\end{proof}

When $\cG$ is a family of subgroups of $\Gamma$, a subset $\cF \subset \Gamma$ is said to be \emph{small relative to $\cG$} if $\cF$ is contained in the union of finitely many subsets of the form $g \Sigma h$ with $g,h \in \Gamma$ and $\Sigma \in \cG$. We always tacitly assume that $\cG$ contains the trivial subgroup $\{e\}$, so that finite subsets of $\Gamma$ always are small relative to $\cG$. When $K$ is a normed space and $f : \Gamma \recht K$, we say that
\begin{itemize}
\item $\dis \lim_{k \recht \infty / \cG} f(k) = 0$ if for every $\eps > 0$, the set $\{k \in \Gamma \mid \|f(k)\| > \eps\}$ is small relative to $\cG$~;
\item $f : \Gamma \recht K$ is \emph{proper relative to $\cG$,} if for every $\kappa > 0$ the set $\{k \in \Gamma \mid \|f(k)\| < \kappa\}$ is small relative to $\cG$.
\end{itemize}

We denote by $\Prob \Gamma$ the set of probability measures on a (countable) group $\Gamma$. We identify $\Prob \Gamma$ with the natural convex subset of $\ell^1(\Gamma)$ and use the $1$-norm on $\ell^1(\Gamma)$. If $g \in \Gamma$ and $\mu \in \Prob \Gamma$, we denote by $g \cdot \mu$ the left translation of $\mu$ by $g$.

\begin{definition}[{\cite[Definition 15.1.2]{BO08}}]\label{def.biexact-G}
A countable group $\Gamma$ with a family of subgroups $\cG$ is said to be \emph{bi-exact relative to $\cG$} if $\Gamma$ is exact and if there exists a map $\mu : \Gamma \recht \Prob \Gamma$ such that
\begin{equation}\label{eq.almost-equiv-one}
\lim_{k \recht \infty / \cG} \|\mu(g k h) - g \cdot \mu(k) \|_1 = 0 \;\;\text{for all}\;\; g,h \in \Gamma \; .
\end{equation}
By definition, a group is bi-exact if and only if it is bi-exact relative to $\{\{e\}\}$.
\end{definition}

As observed by Ozawa, a group is bi-exact if and only if it is exact and belongs to the class $\QHreg$ of \cite{CS11}. We actually have the following more general result, parts of which were already proven in \cite{CS11,CSU11}. For the sake of completeness, we give a detailed proof, using the methods of \cite[Chapter 15]{BO08}. Note however that we do not use this result in the rest of the paper.

\begin{proposition}\label{prop.biexact}
Let $\Gamma$ be a countable group and $\cG$ a family of subgroups of $\Gamma$ with $\{e\} \in \cG$. The following statements are equivalent.
\begin{enumerate}
\item\label{een} There exists a map $\mu : \Gamma \recht \Prob \Gamma$ satisfying \eqref{eq.almost-equiv-one} in Definition \ref{def.biexact-G}.
\item\label{twee} There exist a map $c : \Gamma \recht \ell^2_\R(\Gamma)$ that is proper relative to $\cG$ and that satisfies
\begin{equation}\label{eq.almost-equiv-two}
\sup_{k \in \Gamma} \|c(gkh) - \lambda_g c(k)\|_2 < \infty \;\;\text{for all}\;\; g,h \in \Gamma \; .
\end{equation}
\item\label{drie} There exists an orthogonal representation $\eta : \Gamma \recht \cO(K_\R)$ that is weakly contained in the regular representation and a map $c : \Gamma \recht K_\R$ that is proper relative to $\cG$ and satisfies
\begin{equation}\label{eq.almost-equiv-three}
\sup_{k \in \Gamma} \|c(gkh) - \eta_g c(k)\| < \infty \;\;\text{for all}\;\; g,h \in \Gamma \; .
\end{equation}
\end{enumerate}
In particular, $\Gamma$ is bi-exact relative to $\cG$ if and only if $\Gamma$ is exact and $\Gamma$ satisfies the above equivalent conditions.
\end{proposition}

Note that the class $\QHreg$ of \cite{CS11} is defined as the class of groups $\Gamma$ that satisfy \ref{drie} w.r.t.\ $\cG = \{\{e\}\}$. So, we indeed have that $\cS = \QHreg \cap \{\text{exact}\}$.

\begin{proof}
\ref{een} $\Rightarrow$ \ref{twee}. Assume that $\mu : \Gamma \recht \Prob \Gamma$ satisfies \eqref{eq.almost-equiv-one}. Define $\zeta : \Gamma \recht \ell^2_\R(\Gamma)$ given by $\zeta(k) := \mu(k)^{1/2}$. Note that $\|\zeta(k)\|_2 = 1$ for all $k \in \Gamma$ and that
\begin{equation}\label{eq.equiv-zeta}
\lim_{k \recht \infty / \cG} \|\zeta(g k h) - \lambda_g \zeta(k) \|_2 \leq \lim_{k \recht \infty / \cG} \|\mu(gkh) - g \cdot \mu(k)\|_1^{1/2} = 0 \;\;\text{for all}\;\; g,h \in \Gamma \; .
\end{equation}
Let $\{e\} = E_0 \subset E_1 \subset E_2 \subset \cdots$ be finite subsets of $\Gamma$ such that $E_n^{-1} = E_n$ for all $n$ and $\bigcup_{n=0}^\infty E_n = \Gamma$. Inductively define the subsets $F_n \subset \Gamma$ given by $F_0 = \{e\}$ and, for all $n \geq 1$,
$$F_n := E_n F_{n-1} E_n \; \cup \; \bigcup_{g,h \in E_n} \bigl\{k \in \Gamma \; \big| \; \|\zeta(gkh) - \lambda_g \zeta(k)\|_2 > \frac{1}{n} \bigr\} \;  .$$
By construction, the sets $F_n$ are small relative to $\cG$. Also, the subsets $F_n$ are increasing and their union equals $\Gamma$, because $E_n \subset F_n$. So we can uniquely define the map
$$c : \Gamma \recht \ell^2_\R(\Gamma) : c(k) = \begin{cases} 0 &\;\;\text{if}\;\; k \in F_1 \; , \\ n \; \zeta(k) &\;\;\text{if}\;\; n \geq 1 \;\;\text{and}\;\; k \in F_{n+1} - F_n \; .\end{cases}$$
Whenever $k \in \Gamma - F_n$, we have $\|c(k)\|_2 \geq n$. So, $c$ is proper relative to $\cG$. We prove that $c$ satisfies \eqref{eq.almost-equiv-two}. So fix $g, h \in \Gamma$. Take $m \geq 1$ such that $g,h \in E_m$. It suffices to prove that
\begin{equation}\label{eq.suffices}
\|c(gkh) - \lambda_g c(k) \|_2 \leq 2m \quad\text{for all}\;\; k \in \Gamma \; .
\end{equation}
We first prove \eqref{eq.suffices} if $k \in \Gamma - F_m$. So, $k \in F_{n+1} - F_n$ for some $n \geq m$. Hence $c(k) = n \zeta(k)$. Since $g,h \in E_m \subset E_n = E_n^{-1}$, we also get that $gkh \in F_{n+2} - F_{n-1}$. So $c(gkh)$ can be $n+1$ times, or $n$ times, or $n-1$ times $\zeta(k)$. In all cases $\|c(gkh) - n \zeta(gkh)\|_2 \leq 1$. Since $k \not\in F_n$ and $g,h \in E_m \subset E_n$, we have $\|\zeta(gkh) - \lambda_g \zeta(k)\|_2 \leq 1/n$. Multiplying by $n$, we get that
$$\|c(gkh) - \lambda_g c(k)\|_2 \leq \|c(gkh) - n \, \zeta(gkh)\|_2 + n \, \|\zeta(gkh) - \lambda_g \zeta(k)\|_2 \leq 2 \leq 2m \; .$$
So \eqref{eq.suffices} is proven for all $k \in \Gamma - F_m$. If $k \in F_m$, we have $\|c(k)\|_2 \leq m-1$. Since $g,h \in E_m$, we also have $gkh \in F_{m+1}$ and hence $\|c(gkh)\|_2 \leq m$. Combining both we get that
$$\|c(gkh) - \lambda_g c(k)\|_2 \leq \|c(gkh)\|_2 + \|c(k)\|_2 \leq 2m-1 < 2m \; .$$
So \eqref{eq.suffices} is proven and hence \ref{twee} holds.

\ref{twee} $\Rightarrow$ \ref{drie} is trivial by taking $\eta$ to be the regular representation.

\ref{drie} $\Rightarrow$ \ref{een}. For a finite group $\Gamma$ all statements in the proposition are trivially true (and rather silly). So we assume that $\Gamma$ is a countably infinite group satisfying \ref{drie} and we prove that $\Gamma$ satisfies \ref{een}. Take $\eta : \Gamma \recht \cO(K_\R)$ and $c : \Gamma \recht K_\R$ as in \ref{drie}. Replacing $K_\R$ by the closed linear span of $\{\eta_g c(k) \mid g,k \in \Gamma\}$, we may assume that $K_\R$ is separable. Let $\zeta_0 \in K_\R$ be an arbitrary unit vector and define
$$\zeta : \Gamma \recht K_\R : \zeta(k) = \begin{cases} \|c(k)\|^{-1} c(k) &\;\;\text{if}\;\; c(k) \neq 0 \; , \\ \zeta_0  &\;\;\text{if}\;\; c(k) = 0 \; . \end{cases}$$
By construction $\|\zeta(k)\| = 1$ for all $k \in \Gamma$.
Since for all nonzero vectors $\xi$ and $\xi'$ in a Hilbert space, one has
$$\Bigl\| \frac{\xi}{\|\xi\|} - \frac{\xi'}{\|\xi'\|} \Bigr\| \leq 2 \frac{\|\xi-\xi'\|}{\|\xi'\|} \; ,$$
the properness of $c$ relative to $\cG$ together with \eqref{eq.almost-equiv-three} implies that
\begin{equation}\label{eq.almost-equiv-four}
\lim_{k \recht \infty / \cG} \|\zeta(g k h) - \eta_g \zeta(k) \| = 0 \;\;\text{for all}\;\; g,h \in \Gamma \; .
\end{equation}
Denote by $K$ the complexification of $K_\R$ and still denote by $\eta : \Gamma \recht \cU(K)$ the complexified representation. Denote $\Ktil = K \ot \ell^2(\Gamma)$ and $\Gammatil = \Gamma \times \Gamma$. Consider the unitary representation $\etatil : \Gammatil \recht \cU(\Ktil)$ given by $\etatil_{(g,h)} = \eta_g \ot \lambda_g \rho_h$. Since $\eta$ is weakly contained in the regular representation of $\Gamma$, we get that $\etatil$ is weakly contained in the representation $(g,h) \mapsto \lambda_g \ot \lambda_g \rho_h$ that in turn is unitarily equivalent with the regular representation of $\Gammatil$. So we get a unital $*$-homomorphism $\theta : \Cstarred(\Gammatil) \recht \B(\Ktil)$ satisfying $\theta(\lambda_s) = \etatil_s$ for all $s \in \Gammatil$. Since $\Gammatil$ is an infinite group, $\Cstarred(\Gammatil) \cap \cK(\ell^2(\Gammatil)) = \{0\}$. So by Voiculescu's theorem (see e.g.\ \cite[Corollary II.5.5]{Da96}), there exists a unitary operator $V : \ell^2(\Gammatil) \oplus \Ktil \recht \ell^2(\Gammatil)$ such that
\begin{equation}\label{eq.compact}
a V - V (a \oplus \theta(a)) \quad\text{is a compact operator for all}\;\; a \in \Cstarred(\Gammatil) \; .
\end{equation}
For $k \in \Gamma$, denote by $\delta_k \in \ell^2(\Gamma)$ the canonical basis vector. Define
$$\zeta' : \Gamma \recht \ell^2(\Gammatil) : \zeta'(k) = V(0 \oplus (\zeta(k) \ot \delta_k)) \quad\text{for all}\;\; k \in \Gamma \; .$$
By construction $\|\zeta'(k)\|_2 = 1$ for all $k \in \Gamma$. We claim that
\begin{equation}\label{eq.almost-equiv-five}
\lim_{k \recht \infty / \cG} \|\zeta'(g k h^{-1}) - (\lambda_g \ot \lambda_h) \zeta'(k) \| = 0 \;\;\text{for all}\;\; g,h \in \Gamma \; .
\end{equation}
To prove this claim, fix $g, h \in \Gamma$. Put
$$T := \lambda_{(g,h)} V - V (\lambda_{(g,h)} \oplus \etatil_{(g,h)}) \; .$$
From \eqref{eq.compact}, we know that $T$ is a compact operator. If $k \recht \infty/\cG$, certainly $k \recht \infty$ and hence $\zeta(k) \ot \delta_k$ tends to $0$ weakly. Since $T$ is compact, it follows that
$$\lim_{k \recht \infty/\cG} \| T(0 \oplus (\zeta(k) \ot \delta_k))\|_2 = 0 \; .$$
This precisely means that
\begin{equation}\label{eq.previous}
\lim_{k \recht \infty/\cG} \bigl\|(\lambda_g \ot \lambda_h) \zeta'(k) - V \bigl( 0 \oplus \bigl( \eta_g \zeta(k) \ot \delta_{gkh^{-1}} \bigr) \bigr)\bigr\|_2 = 0 \; .
\end{equation}
From \eqref{eq.almost-equiv-four}, it follows that
$$\lim_{k \recht \infty/\cG} \bigl\| (\eta_g \zeta(k) \ot \delta_{gkh^{-1}}) \; - \; (\zeta(gkh^{-1}) \ot \delta_{gkh^{-1}}) \bigr\| = 0 \; .$$
In combination with \eqref{eq.previous}, we exactly get the claim \eqref{eq.almost-equiv-five}.

To every unit vector $\zeta \in \ell^2(\Gammatil)$, we associate the probability measure $\cT(\zeta) \in \Prob \Gamma \subset \ell^1(\Gamma)$ given by
$$\bigl(\cT(\zeta)\bigr)(s) = \sum_{t \in \Gamma} |\zeta(s,t)|^2 \quad\text{for all}\;\; s \in \Gamma \; .$$
Clearly, $\cT((\lambda_g \ot \lambda_h) \zeta) = g \cdot \cT(\zeta)$ and, using the Cauchy-Schwarz inequality, $\|\cT(\zeta_1) - \cT(\zeta_2)\|_1 \leq 2 \|\zeta_1 - \zeta_2\|_2$ for all unit vectors $\zeta_1, \zeta_2 \in \ell^2(\Gammatil)$. Defining $\mu : \Gamma \recht \Prob \Gamma$ by $\mu(k) = \cT(\zeta'(k))$ for all $k \in \Gamma$, we then get that \eqref{eq.almost-equiv-five} implies \eqref{eq.almost-equiv-one}. So we have proven that \ref{een} holds.
\end{proof}

We record the following lemma from \cite{BO08}.

\begin{lemma}[{\cite[Lemma 15.3.3]{BO08}}]\label{lem.biexact-products}
Let $\Gamma_1,\ldots,\Gamma_n$ be countable groups with families of subgroups $\cG_1,\ldots,\cG_n$. Denote by $\Gamma = \Gamma_1 \times \cdots \times \Gamma_n$ the direct product group. If every $\Gamma_i$ is bi-exact relative to $\cG_i$, then $\Gamma$ is bi-exact relative to
$$\cG = \bigcup_{i=1}^n \Bigl\{ \Lambda \times \prod_{j \neq i} \Gamma_j \; \Big| \; \Lambda \in \cG_i \Bigr\} \; .$$
\end{lemma}
\begin{proof}
Clearly $\Gamma$ is an exact group. Take maps $\mu_i : \Gamma_i \recht \Prob \Gamma_i$ satisfying \eqref{eq.almost-equiv-one} in Definition \ref{def.biexact-G}. Then the map
$$\mu : \Gamma \recht \Prob \Gamma : \mu(g_1,\ldots,g_n) = \mu_1(g_1) \times \cdots \times \mu_n(g_n)$$
satisfies the same condition.
\end{proof}

\section{Key theorem}

We prove the following key theorem from which all other results in the paper will be deduced.

\begin{theorem}\label{thm.key}
Let $\Gamma$ be a weakly amenable group that is bi-exact relative to a family $\cG$ of subgroups of $\Gamma$. Assume that $\Gamma \actson (B,\tau)$ is any trace preserving action on an arbitrary tracial von Neumann algebra $(B,\tau)$. Put $M = B \rtimes \Gamma$. Let $q \in M$ be a projection and $A \subset qMq$ a von Neumann subalgebra that is amenable relative to $B$. Denote by $P = \cN_{qMq}(A)\dpr$ the normalizer of $A$ inside $qMq$. Then at least one of the following statements holds.
\begin{itemize}
\item $P$ is amenable relative to $B$.
\item There exists a $\Sigma \in \cG$ such that $A \prec_M B \rtimes \Sigma$.
\end{itemize}
\end{theorem}

\subsection{It suffices to prove Theorem \ref{thm.key} for the trivial action}

\begin{proposition}\label{prop.trivial-action}
If Theorem \ref{thm.key} holds for the trivial action on \emph{arbitrary} tracial von Neumann algebras, then Theorem \ref{thm.key} also holds for arbitrary trace preserving actions.
\end{proposition}
\begin{proof}
Assume that Theorem \ref{thm.key} holds for the trivial action on arbitrary tracial von Neumann algebras. Let $\Gamma \actson (B,\tau)$ be any trace preserving action. Put $M = B \rtimes \Gamma$. Let $q \in M$ be a projection and $A \subset qMq$ a von Neumann subalgebra that is amenable relative to $B$. Denote by $P = \cN_{qMq}(A)\dpr$ the normalizer of $A$ inside $qMq$.

Define $\cM := M \ovt \rL \Gamma$ which we view as the crossed product of $\Gamma$ with the trivial action on $M$. Consider the trace preserving embedding
$$\Delta : M \recht \cM : \Delta(b u_g) = b u_g \ot u_g \quad\text{for all}\;\; b \in B , g \in \Gamma \; .$$
Put $\qtil = \Delta(q)$ and $\cA := \Delta(A)$. Denote by $\cP := \cN_{\qtil \cM \qtil}(\cA)\dpr$ the normalizer of $\cA$ inside $\qtil \cM \qtil$. Note that $\Delta(P) \subset \cP$.

As explained in the first paragraphs of the proof of \cite[Lemma 4.1]{PV11}, we have that $\cA$ is amenable relative to $M \ot 1$.
Since Theorem \ref{thm.key} holds for the trivial action of $\Gamma$ on $M$, at least one of the following statements holds.
\begin{itemize}
\item $\Delta(P)$ is amenable relative to $M \ot 1$.
\item There exists a $\Sigma \in \cG$ such that $\cA \prec_{\cM} M \ovt \rL \Sigma$.
\end{itemize}
If $\cA \prec_{\cM} M \ovt \rL \Sigma$, it is easy to check that $A \prec_M B \rtimes \Sigma$ so that the second statement in the formulation of Theorem \ref{thm.key} holds.

Next assume that $\Delta(P)$ is amenable relative to $M \ot 1$. So we have a $\Delta(P)$-central positive functional $\Omtil$ on $\qtil \langle \cM, e_{M \ot 1} \rangle \qtil$ satisfying $(\Omtil \circ \Delta)_{|qMq} = \tau_{|qMq}$. Since $E_{M \ot 1} \circ \Delta = \Delta \circ E_B$, the embedding $\Delta : M \recht \cM$ can be extended to an embedding
$$\Psi : \langle M, e_B \rangle \recht \langle \cM , e_{M \ot 1} \rangle : \Psi(e_B) = e_{M \ot 1}\quad\text{and}\quad \Psi(x) = \Delta(x) \quad\text{for all}\;\; x \in M \; .$$
It follows that $\Omtil \circ \Psi$ is a $P$-central positive functional on $q \langle M,e_B \rangle q$ satisfying $\Om_{|qMq} = \tau_{|qMq}$. Hence, $P$ is amenable relative to $B$ and the first statement in the formulation of Theorem \ref{thm.key} holds.
\end{proof}

\subsection{Setup and notations for the proof of Theorem \ref{thm.key}}\label{subsec.setup}

By Proposition \ref{prop.trivial-action}, we may assume that $\Gamma \actson (B,\tau)$ is the trivial action. We put $M := B \ovt \rL \Gamma$. For simplicity of notation we assume that $q = 1$. As in \cite[Remark 6.3]{PV11}, this notational simplification is only cosmetic and does not hide any essential parts of the argument.

So we are given a von Neumann subalgebra $A \subset M$ that is amenable relative to $B$. We denote by $P = \cN_M(A)\dpr$ its normalizer.
Following \cite[Theorem 5.1]{PV11}, we define $N$ as the von Neumann algebra generated by $B$ and $P\op$ on the Hilbert space $\rL^2(M) \ot_A \rL^2(P)$. Put $\cN := N \ovt \rL \Gamma$ and define the tautological embeddings
$$\pi : M \recht \cN : \pi(b \ot u_g) = b \ot u_g \quad\text{and}\quad \theta : P\op \recht \cN : \theta(y\op) = y\op \ot 1$$
for all $b \in B$, $g \in \Gamma$ and $y \in P$. Note that $\pi(M)$ and $\theta(P\op)$ commute and that together they generate $\cN$.

By \cite[Theorem 5.1]{PV11} we find a net of normal states $\om_n \in \cN_*$ satisfying the following properties.
\begin{itemize}
\item $\om_n(\pi(x)) \recht \tau(x)$ for all $x \in M$,
\item $\om_n(\pi(a) \theta(\abar)) \recht 1$ for all $a \in \cU(A)$,
\item $\|\om_n \circ \Ad(\pi(u) \theta(\ubar)) - \om_n\| \recht 0$ for all $u \in \cN_M(A)$.
\end{itemize}

We fix a standard Hilbert space $H$ for $N$ and we always view $N$ as acting on $H$. This standard Hilbert space comes with the canonical anti-unitary involution $J$.
Being the tensor product of $N$ and $\rL(\Gamma)$, the von Neumann algebra $\cN$ is standardly represented on $\cH := H \ot \ell^2(\Gamma)$ by the formula
$$(x \ot u_g) \cdot (\xi \ot \delta_h) = x \xi \ot \delta_{gh} \quad\text{for all}\;\; x \in N \; , \; g,h \in \Gamma \; , \; \xi \in H \; .$$
The corresponding anti-unitary involution $\cJ : \cH \recht \cH$ is given by $\cJ(\xi \ot \delta_g) = J \xi \ot \delta_{g^{-1}}$. So the von Neumann algebras $\pi(M)$, $\cJ \pi(M) \cJ$, $\theta(P\op)$ and $\cJ \theta(P\op) \cJ$ all act on $\cH$ and mutually commute.

Denote by $\xi_n \in \cH$ the canonical positive unit vectors that implement the normal states $\om_n$ on $\cN$. Whenever $u \in \cN_M(A)$ it follows from \cite[Theorem IX.1.2.(iii)]{Ta03} that the vector
$$\pi(u) \, \theta(\ubar) \, \cJ \pi(u) \, \theta(\ubar) \cJ \, \xi_n$$
is the canonical positive vector that implements $\om_n \circ \Ad(\pi(u^*)\theta(u\op))$. Using the Powers-St{\o}rmer inequality (see e.g.\ \cite[Theorem IX.1.2.(iv)]{Ta03}), the properties of $(\om_n)$ can now be rewritten as follows in terms of the net $(\xi_n)$.
\begin{align}
& \langle \pi(x) \xi_n ,\xi_n \rangle = \om_n(\pi(x)) \recht \tau(x) \quad\text{for all}\;\; x \in M \; , \label{eq.onx}\\
& \| \pi(a) \theta(\abar) \xi_n - \xi_n\| \recht 0 \quad\text{for all}\;\; a \in \cU(A) \label{eq.ona}\; , \\
& \| \pi(u) \, \theta(\ubar) \, \cJ \pi(u) \, \theta(\ubar) \cJ \, \xi_n \; - \; \xi_n \| \recht 0 \quad\text{for all}\;\; u \in \cN_M(A) \; .\label{eq.onu}
\end{align}

Since $\Gamma$ is bi-exact relative to $\cG$, Definition \ref{def.biexact-G} provides a map $\mu : \Gamma \recht \Prob \Gamma$ such that
$$\lim_{k \recht \infty/\cG} \|\mu(gkh) - g \cdot \mu(k)\|_1 = 0 \quad\text{for all}\;\; g,h \in \Gamma \; .$$
Define $\zeta : \Gamma \recht \ell^2(\Gamma) : \zeta(k) = \mu(k)^{1/2}$. Note that
\begin{equation}\label{eq.zeta}
\|\zeta(k)\|_2 = 1 \;\;\text{for all}\;\; k \in \Gamma \quad\text{and}\quad \lim_{k \recht \infty/\cG} \|\zeta(gkh^{-1}) - \lambda_g \zeta(k)\|_2 = 0 \quad\text{for all}\;\; g,h \in \Gamma \; .
\end{equation}
Define the isometry
$$V : \ell^2(\Gamma) \recht \ell^2(\Gamma) \ot \ell^2(\Gamma) : V \delta_k = \zeta(k) \ot \delta_k \quad\text{for all}\;\; k \in \Gamma \; .$$
We denote by $\cS$ the directed set of subsets of $\Gamma$ that are small relative to $\cG$. For every subset $\cF \subset \Gamma$, we denote by $P_\cF$ the orthogonal projection of $\ell^2(\Gamma)$ onto $\ell^2(\cF)$. Then \eqref{eq.zeta} can be rewritten as
\begin{equation*}
\lim_{\cF \in \cS} \; \bigl\| \, \bigl((\lambda_g \ot \lambda_g \rho_h) V - V \lambda_g \rho_h\bigr) \, P_{\Gamma - \cF} \, \bigr\| = 0 \quad\text{for all}\;\; g, h \in \Gamma \; .
\end{equation*}
The representation $(g,h) \mapsto \lambda_g \ot \lambda_g \rho_h$ of $\Gamma \times \Gamma$ is unitarily conjugate to the regular representation $(g,h) \mapsto \lambda_g \ot \lambda_h$ through the unitary $U \in \cU(\ell^2(\Gamma) \ot \ell^2(\Gamma))$ given by $U(\delta_k \ot \delta_r) = \delta_k \ot \delta_{r^{-1}k}$. We define the isometry
$$W : \ell^2(\Gamma) \recht \ell^2(\Gamma) \ot \ell^2(\Gamma) : W = U V \; .$$
We then get
\begin{equation}\label{eq.zeta-rewrite}
\lim_{\cF \in \cS} \; \bigl\| \, \bigl((\lambda_g \ot \lambda_h) W - W \lambda_g \rho_h\bigr) \, P_{\Gamma - \cF} \, \bigr\| = 0 \quad\text{for all}\;\; g, h \in \Gamma \; .
\end{equation}

Define the weakly dense $*$-subalgebra $M_0 \subset M$ given by $M_0 = B \otalg \C \Gamma$. Then define the unital $*$-algebras
$$\cD := M \otalg M\op \otalg P\op \otalg P \quad\text{with $*$-subalgebra}\quad \cD_0 = M_0 \otalg M_0\op \otalg P\op \otalg P \; .$$
Define the unique $*$-homomorphisms
$$
\Psi : \cD \recht \B(H \ot \ell^2(\Gamma) \ot \ell^2(\Gamma)) \quad\text{and}\quad \Theta : \cD \recht \B(H \ot \ell^2(\Gamma))
$$
that are separately normal in each of the tensor factors of $\cD = M \otalg M\op \otalg P\op \otalg P$ and that satisfy
\begin{align*}
& \Psi\bigl((b \ot u_g) \ot (c \ot u_h)\op \ot y\op \ot z \bigr) = b \;\; Jc^*J \;\; y\op \;\; J \zbar J \; \ot \; \lambda_g \ot \lambda_{h^{-1}} \; ,\\
& \Theta\bigl((b \ot u_g) \ot (c \ot u_h)\op \ot y\op \ot z \bigr) = \pi(b \ot u_g) \;\; \cJ \pi(c \ot u_h)^* \cJ \;\; \theta(y\op) \;\; \cJ \theta(\zbar) \cJ \; ,
\end{align*}
for all $b,c \in B$, $g,h \in \Gamma$, and $y,z \in P$. Note that for a better understanding of the defining formulae of $\Psi$ and $\Theta$, one should identify $P$ with $(P\op)\op$. Also note that by the definition of $\pi$ and $\cJ$, we have
$$\Theta\bigl((b \ot u_g) \ot (c \ot u_h)\op \ot y\op \ot z \bigr) = b \;\; Jc^*J \;\; y\op \;\; J \zbar J \; \ot \; \lambda_g \rho_{h^{-1}} \; ,
$$
for all $b,c \in B$, $g,h \in \Gamma$, and $y,z \in P$.
By linearity, \eqref{eq.zeta-rewrite} thus implies that
\begin{equation}\label{eq.equiv-W}
\lim_{\cF \in \cS} \; \bigl\| \, \bigl(\Psi(S) (1 \ot W) - (1 \ot W) \Theta(S) \bigr) \, (1 \ot P_{\Gamma - \cF}) \, \bigr\| = 0 \quad\text{for all}\;\; S \in \cD_0 \; .
\end{equation}
Since $W$ is an isometry, we get in particular that
\begin{equation}\label{eq.est-norm-Theta}
\limsup_{\cF \in \cS} \; \bigl\| \, \Theta(S) \, (1 \ot P_{\Gamma - \cF}) \, \bigr\| \; \leq \; \|\Psi(S)\| \quad\text{for all}\;\; S \in \cD_0 \; .
\end{equation}
It is important to note that \eqref{eq.equiv-W} and \eqref{eq.est-norm-Theta} only hold for $S \in \cD_0$ and not necessarily for all $S \in \cD$.

\subsection{The proof of Theorem \ref{thm.key} splits up in two cases}

We get the following dichotomy in terms of the net of unit vectors $(\xi_n)$ in $H \ot \ell^2(\Gamma)$ that we introduced in the previous section.

{\bf Case 1.} For every subset $\cF \subset \Gamma$ that is small relative to $\cG$, we have
$$\lim_n \| (1 \ot P_\cF) \xi_n\| = 0 \; .$$

{\bf Case 2.} There exists a subset $\cF \subset \Gamma$ that is small relative to $\cG$ and that satisfies
$$\limsup_n \| (1 \ot P_\cF) \xi_n\| > 0 \; .$$

\subsection{Proof of Theorem \ref{thm.key} in case 1}\label{subsec.case1}

Choose a (typically non-normal) state $\Om_1$ on $\B(H \ot \ell^2(\Gamma))$ as a weak$^*$ limit point of the net of states $S \mapsto \langle S \xi_n,\xi_n\rangle$. From \eqref{eq.onx} and \eqref{eq.onu}, we get that
\begin{align}
& \Om_1(\pi(x)) = \tau(x) \quad\text{and}\quad |\Om_1(S \pi(x))| \leq \|S\| \, \|x\|_2 \quad\text{for all}\;\; x \in M , S \in \B(H \ot \ell^2(\Gamma)) \; ,\label{eq.aaa}\\
& \Om_1(\Theta(u \ot \ubar \ot \ubar \ot u)) = 1 \quad\text{for all}\;\; u \in \cN_M(A) \; .\label{eq.bbb}
\end{align}
Since by assumption $\lim_n \| (1 \ot P_\cF) \xi_n\| = 0$ for all subsets $\cF \subset \Gamma$ that are small relative to $\cG$, we also get that
$$\Om_1(S) = \Om_1(S(1 \ot P_{\Gamma - \cF}))$$
for all $S \in \B(H \ot \ell^2(\Gamma))$ and all subsets $\cF \subset \Gamma$ that are small relative to $\cG$. In combination with \eqref{eq.est-norm-Theta}, it follows that
\begin{multline}\label{eq.thatsitpartial}
|\Om_1(\Theta(S))| = \limsup_{\cF \in \cS} |\Om_1(\Theta(S)(1 \ot P_{\Gamma - \cF}))| \leq \limsup_{\cF \in \cS} \| \Theta(S)(1 \ot P_{\Gamma - \cF}) \|
\\ \leq \|\Psi(S)\| \quad\text{for all}\;\; S \in \cD_0 \; .
\end{multline}
The main point of the proof will now be to prove the existence of $\kappa > 0$ such that
\begin{equation}\label{eq.maineq}
|\Om_1(\Theta(S))| \leq \kappa^2 \, \|\Psi(S)\| \quad\text{for all}\;\; S \in \cD \; .
\end{equation}

Since $\Gamma$ is weakly amenable, choose a sequence of finitely supported Herz-Schur multipliers $f_i : \Gamma \recht \C$ such that $f_i \recht 1$ pointwise and $\limsup_i \|f_i\|\cb = \kappa < \infty$. Denote by $\m_i : \rL \Gamma \recht \rL \Gamma$ the normal completely bounded maps given by $\m_i(u_g) = f_i(g) u_g$ for all $g \in \Gamma$. We define the corresponding normal completely bounded maps $\vphi_i : M \recht M$ and $\vphitil_i : M\op \recht M\op$ given by
$$\vphi_i(b \ot u_g) = f_i(g) (b \ot u_g) \quad\text{and}\quad \vphitil_i(\overline{b \ot u_g}) = f_i(g) \, \overline{b \ot u_g}$$
for all $b \in B$ and $g \in \Gamma$.

Observe that for all $x \in M$, we have $\lim_i \|x - \vphi_i(x)\|_2 = 0$ and $\lim_i \|x\op - \vphitil_i(x\op)\|_2 = 0$. Since the functions $f_i$ are finitely supported, we also note that for all $S \in \cD$, we have
$$(\vphi_i \ot \vphitil_i \ot \id \ot \id)(S) \in \cD_0 \; .$$
We claim that for all $x_1,x_2 \in M$ and all $y,z \in P$, we have
\begin{equation}\label{eq.crucialclaim}
\lim_i \Om_1\bigl(\Theta\bigl( \vphi_i(x_1) \ot \vphitil_i(x_2\op) \ot y\op \ot z \bigr)\bigr) = \Om_1\bigl(\Theta(x_1 \ot x_2\op \ot y\op \ot z)\bigr) \; .
\end{equation}
To prove this claim, note that \eqref{eq.aaa} implies that
\begin{align*}
\limsup_i \;\; \Bigl| \Om_1\bigl( \Theta\bigl( & (\vphi_i(x_1) - x_1) \ot \vphitil_i(x_2\op) \ot y\op \ot z \bigr)\bigr) \Bigr| \\
& \leq \limsup_i \;\; \bigl\| \Theta(1 \ot \vphitil_i(x_2\op) \ot y\op \ot z) \bigr\| \; \|\vphi_i(x_1) - x_1 \|_2 \\
& \leq \limsup_i \;\; \|\vphitil_i(x_2\op)\| \; \|y\| \; \|z\| \; \|\vphi_i(x_1) - x_1 \|_2 \\
& \leq \kappa \; \|x_2\| \; \|y\| \; \|z\| \; \limsup_i \|\vphi_i(x_1) - x_1 \|_2 = 0 \; .
\end{align*}
One similarly proves that
$$\lim_i \Om_1\bigl( \Theta\bigl( x_1 \ot (\vphitil_i(x_2\op) - x_2) \ot y\op \ot z \bigr)\bigr) = 0 \; .$$
Summing up both, the claim \eqref{eq.crucialclaim} follows. By linearity, we get that
\begin{equation}\label{eq.crucialtwo}
\Om_1(\Theta(S)) = \lim_i \Om_1\bigl(\Theta\bigl( (\vphi_i \ot \vphitil_i \ot \id \ot \id)(S) \bigr)\bigr) \quad\text{for all}\;\; S \in \cD \; .
\end{equation}
We are now ready to prove \eqref{eq.maineq}. Observe that $\Psi(\cD) \subset \B(H) \ovt \rL \Gamma \ovt \rL \Gamma$ and that
$$\Psi\bigl( (\vphi_i \ot \vphitil_i \ot \id \ot \id)(S) \bigr) = (\id \ot \m_i \ot \m_i) \bigl(\Psi(S)\bigr) \quad\text{for all}\;\; S \in \cD \; .$$
In combination with \eqref{eq.crucialtwo} and \eqref{eq.thatsitpartial}, we get for all $S \in \cD$ that
\begin{align*}
|\Om_1(\Theta(S))| & = \limsup_i \bigl|\Om_1\bigl(\Theta\bigl( (\vphi_i \ot \vphitil_i \ot \id \ot \id)(S) \bigr)\bigr)\bigr| \\
& \leq \limsup_i \bigl\|\Psi\bigl( (\vphi_i \ot \vphitil_i \ot \id \ot \id)(S)  \bigr) \bigr\| \\
& = \limsup_i \bigl\| (\id \ot \m_i \ot \m_i) \bigl( \Psi(S) \bigr) \bigr\| \\
& \leq \limsup_i \|\m_i\|\cb^2 \, \|\Psi(S)\| \leq \kappa^2 \, \|\Psi(S)\| \; .
\end{align*}
So, \eqref{eq.maineq} is proven.

Define the unital C$^*$-algebra $\cQ \subset \B(H \ot \ell^2(\Gamma) \ot \ell^2(\Gamma))$ as the norm closure of $\Psi(\cD)$. Because of \eqref{eq.maineq}, there is a unique continuous functional $\Om_2 \in \cQ^*$ such that $\Om_2(\Psi(S)) = \Om_1(\Theta(S))$ for all $S \in \cD$. Since $\Om_1$ is positive, it follows that for all $S \in \cD$,
$$\Om_2(\Psi(S)^* \Psi(S)) = \Om_2(\Psi(S^* S)) = \Om_1(\Theta(S^* S)) = \Om_1(\Theta(S)^* \Theta(S)) \geq 0 \; .$$
By density, it follows that $\Om_2(T^* T) \geq 0$ for all $T \in \cQ$. So, $\Om_2$ is a positive functional on $\cQ$. Since $\Om_2(1) = 1$, we conclude that $\Om_2$ is a state on $\cQ$.

Denote $\pi_0 : M \recht \B(H \ot \ell^2(\Gamma)) : \pi_0(b \ot u_g) = b \ot \lambda_g$ and note that $\pi_0(x) \ot 1 = \Psi(x \ot 1 \ot 1 \ot 1)$ for all $x \in M$.
From \eqref{eq.aaa} and \eqref{eq.bbb}, we get that
\begin{equation}\label{eq.verymuch}
\Om_2(\pi_0(x)\ot 1) = \tau(x)\; , \;\forall x \in M \quad\text{and}\quad \Om_2(\Psi(u \ot \ubar \ot \ubar \ot u)) = 1 \; , \;\forall u \in \cN_M(A) \; .
\end{equation}
By the Hahn-Banach theorem, we can extend $\Om_2$ to a functional on $\B(H \ot \ell^2(\Gamma) \ot \ell^2(\Gamma))$ without increasing the norm of $\Om_2$. We still denote this extension by $\Om_2$. Since $\|\Om_2\| = 1 = \Om_2(1)$, we get that the extended $\Om_2$ is still a state. Since the state $\Om_2$ equals $1$ on the unitaries $\Psi(u \ot \ubar \ot \ubar \ot u)$, $u \in \cN_M(A)$, we get that
\begin{multline}\label{eq.invariance}
\Om_2(S \Psi(u \ot \ubar \ot \ubar \ot u)) = \Om_2(S) = \Om_2(\Psi(u \ot \ubar \ot \ubar \ot u) S) \\ \text{for all}\;\; S \in \B(H \ot \ell^2(\Gamma) \ot \ell^2(\Gamma))\;\;\text{and all}\;\; u \in \cN_M(A) \; .
\end{multline}

Define the state $\Om \in B \ovt \B(\ell^2(\Gamma))$ given by $\Om(S) = \Om_2(S \ot 1)$. Since $\Psi( 1 \ot M\op \ot P\op \ot P)$ commutes with $B \ovt \B(\ell^2(\Gamma)) \ovt 1$, it follows from \eqref{eq.invariance} that $\Om$ is a $\pi_0(\cN_M(A))$-central state on $B \ovt \B(\ell^2(\Gamma))$. From \eqref{eq.verymuch}, we get that $\Om(\pi_0(x)) = \tau(x)$ for all $x \in M$.

We claim that $\Om$ is actually $\pi_0(P)$-central. Fix $S \in B \ovt \B(\ell^2(\Gamma))$. Since $\Om \circ \pi_0 = \tau$, it follows from the Cauchy-Schwarz inequality that
$$|\Om(S \pi_0(x))| \leq \|S\| \, \|x\|_2 \quad\text{and}\quad |\Om(\pi_0(x) S)| \leq \|S \| \, \|x\|_2 \quad\text{for all}\;\; x \in M \; .$$
So, the set of $x \in M$ satisfying $\Om(S \pi_0(x)) = \Om(\pi_0(x) S)$ is a $\|\,\cdot\,\|_2$-closed vector subspace of $M$. Since it contains $\cN_M(A)$, it also contains $P = \cN_M(A)\dpr$. This proves the claim that $\Om$ is a $\pi_0(P)$-central state.

The inclusion $\pi_0 : M \recht B \ovt \B(\ell^2(\Gamma))$ is canonically isomorphic with the inclusion $M \subset \langle M, e_{B \ot 1} \rangle$. So, we have found a $P$-central state on $\langle M,e_{B \ot 1} \rangle$ whose restriction to $M$ equals $\tau$. This means that $P$ is amenable relative to $B \ot 1$ and concludes the proof of Theorem \ref{thm.key} in case~1.

\subsection{Proof of Theorem \ref{thm.key} in case 2}

Take $\delta_1 > 0$ and take a subset $\cF \subset \Gamma$ that is small relative to $\cG$ and that satisfies
$$\limsup_n \| (1 \ot P_\cF) \xi_n\| > \delta_1 \; .$$
Since $\cF$ is small relative to $\cG$, we have that $\cF$ is contained in the union of $m < \infty$ subsets of $\Gamma$ of the form $g_0 \Sigma_0 h_0$ with $g_0,h_0 \in \Gamma$ and $\Sigma_0 \in \cG$. Putting $\delta = \delta_1 / m$, we find $g_0,h_0 \in \Gamma$ and $\Sigma_0 \in \cG$ such that
\begin{equation*}
\limsup_n \|(1 \ot P_{g_0 \Sigma_0 h_0}) \xi_n \| > \delta \; .
\end{equation*}
Put $\Sigma = h_0^{-1} \Sigma_0 h_0$ and denote by $\cF_0$ the singleton $\{g_0 h_0\}$. Replacing $(\xi_n)$ by a subnet it follows that $\cF_0$ is a finite subset of $\Gamma$ satisfying
\begin{equation}\label{eq.bigger}
\liminf_n \|(1 \ot P_{\cF_0 \Sigma}) \xi_n \| > \delta \; .
\end{equation}
We will show that $A \prec_M B \ovt \rL \Sigma$, using an argument inspired by the proof of \cite[Lemma 6.2]{CSU11}. Since $\Sigma$ is a conjugate of $\Sigma_0$ it then also follows that $A \prec_M B \ovt \rL \Sigma_0$. Since $\Sigma_0 \in \cG$, this will conclude the proof of Theorem \ref{thm.key} in case 2.

Assume that $A \not\prec_M B \ovt \rL \Sigma$. We will deduce below that for any finite subset $\cF_0 \subset \Gamma$ and any $\delta > 0$ satisfying \eqref{eq.bigger}, there exists a larger finite subset $\cF_1 \subset \Gamma$ such that
\begin{equation}\label{eq.evenbigger}
\liminf_n \|(1 \ot P_{\cF_1 \Sigma}) \xi_n \| > \sqrt{2} \delta  \; .
\end{equation}
Take an integer $k$ such that $2^{k/2} \delta > 1$. Iterating the above procedure $k$ times, we find a finite subset $\cF_k \subset \Gamma$ that satisfies the absurd statement
$$1 = \lim_n \|\xi_n \| \geq \liminf_n \|(1 \ot P_{\cF_k \Sigma}) \xi_n\| > 2^{k/2} \delta  > 1 \; .$$
So it remains to find a finite subset $\cF_1 \subset \Gamma$ satisfying \eqref{eq.evenbigger}.

Following \cite[Formula (6.9)]{CSU11}, we first claim that
\begin{equation}\label{eq.est-with-F0}
\limsup_n \|\pi(x) (1 \ot P_{\cF_0 \Sigma}) \xi_n\| \leq |\cF_0| \, \|x\|_2 \quad\text{for all}\;\; x \in M \; .
\end{equation}
To prove this claim it suffices to check that for all $g \in \Gamma$ and $x \in M$ we have
\begin{equation}\label{eq.suffices-for-est-with-F0}
\limsup_n \|\pi(x) (1 \ot P_{g \Sigma}) \xi_n \| \leq \|x\|_2 \; .
\end{equation}
First observe that
$$1 \ot P_{g \Sigma} = \pi(1 \ot u_g) (1 \ot P_\Sigma) \pi(1 \ot u_g^*) \; .$$
It then follows that, writing $y = (1 \ot u_g^*) x^* x (1 \ot u_g)$, we have
\begin{align*}
\|\pi(x)  (1 \ot P_{g \Sigma}) \xi_n \|^2 &= \langle (1 \ot P_\Sigma) \pi\bigl((1 \ot u_g^*) x^* x (1 \ot u_g)\bigr) (1 \ot P_\Sigma) \pi(1 \ot u_g^*) \xi_n, \pi(1 \ot u_g^*) \xi_n \rangle \\
&= \langle (1 \ot P_\Sigma)\pi(E_{B \ovt \rL \Sigma}(y) (1 \ot u_g^*)) \xi_n, \pi(1 \ot u_g^*) \xi_n \rangle \\
& = \|(1 \ot P_\Sigma)\pi(E_{B \ovt \rL \Sigma}(y)^{1/2} (1 \ot u_g^*)) \xi_n \|^2 \; .
\end{align*}
Using \eqref{eq.onx}, we conclude that
$$\limsup_n \|\pi(x) (1 \ot P_{g \Sigma}) \xi_n \|^2 \leq \| E_{B \ovt \rL \Sigma}(y)^{1/2} (1 \ot u_g^*) \|_2^2 = \tau(y) = \|x\|_2^2 \; .$$
This establishes \eqref{eq.suffices-for-est-with-F0}. Hence also the claim \eqref{eq.est-with-F0} follows.

Because of \eqref{eq.bigger} we can take $\eps > 0$ such that
\begin{equation}\label{eq.weerschat}
\limsup_n \|\xi_n - (1 \ot P_{\cF_0 \Sigma}) \xi_n \| < \sqrt{1-\delta^2} - \eps \; .
\end{equation}
For every $x \in M$ we denote by $x = \sum_{g \in \Gamma} (x)_g \ot u_g$, with $(x)_g \in B$, the Fourier decomposition of $x$. We claim that there exist $a \in \cU(A)$ and $v \in B \otalg \C \Gamma$ such that
\begin{equation}\label{eq.claimv}
\|a-v\|_2 < \frac{\eps}{|\cF_0|} \quad\text{and}\quad (v)_g = 0 \;\;\text{for all}\;\; g \in \cF_0  \Sigma \cF_0^{-1} \; .
\end{equation}
To prove this claim, first take $a \in \cU(A)$ such that
$$\bigl\| E_{B \ovt \rL \Sigma}((1 \ot u_g^*) a (1 \ot u_h))\bigr\|_2 < \frac{\eps}{3 |\cF_0|^3} \quad\text{for all}\;\; g,h \in \cF_0 \; .$$
This is possible by our assumption that $A \not\prec_M B \ovt \rL \Sigma$. For any subset $F \subset \Gamma$ we also denote by $1 \ot P_F$ the orthogonal projection of $\rL^2(M)$ onto the closure of $\lspan \{b \ot u_g \mid b \in B , g \in F\}$. So we have chosen the unitary $a \in \cU(A)$ such that $\|(1 \ot P_{g \Sigma h^{-1}})(a)\|_2 < \eps / (3 |\cF_0|^3)$ for all $g,h \in \cF_0$. It follows that $\|(1 \ot P_{\cF_0 \Sigma \cF_0^{-1}})(a)\|_2 < \eps / (3 |\cF_0|)$.
Choose $a' \in B \otalg \C \Gamma$ such that $\|a-a'\|_2 < \eps/ (3 |\cF_0|)$. It follows that $\|(1 \ot P_{\cF_0 \Sigma \cF_0^{-1}})(a')\|_2 < 2 \eps / (3 |\cF_0|)$. Defining $v := a' - (1 \ot P_{\cF_0 \Sigma \cF_0^{-1}})(a')$, the elements $a \in \cU(A)$ and $v \in B \otalg \C \Gamma$ satisfy claim \eqref{eq.claimv}.

From \eqref{eq.ona}, we know that $\lim_n \|\xi_n - \pi(a^*)\theta(a\op) \xi_n\| = 0$. In combination with \eqref{eq.weerschat} it follows that
\begin{multline}\label{eq.tussenstapje}
\limsup_n \|  \xi_n - \theta(\abar) \pi(a) \, (1 \ot P_{\cF_0 \Sigma})  \xi_n \| = \limsup_n \|\pi(a^*) \theta(a\op) \xi_n - (1 \ot P_{\cF_0 \Sigma}) \xi_n \| \\
 = \limsup_n \| \xi_n - (1 \ot P_{\cF_0 \Sigma}) \xi_n \|  < \sqrt{1-\delta^2} - \eps \; .
\end{multline}
Since $\|a - v\|_2 < \eps / |\cF_0|$, it follows from \eqref{eq.est-with-F0} that
$$\limsup_n \|\pi(a-v) \, (1 \ot P_{\cF_0 \Sigma}) \xi_n \| < \eps \; .$$
In combination with \eqref{eq.tussenstapje}, we get that
\begin{equation}\label{eq.thatsit}
\limsup_n \|\xi_n - \theta(\abar) \pi(v) \, (1 \ot P_{\cF_0 \Sigma}) \xi_n \| < \sqrt{1-\delta^2} \; .
\end{equation}
Define the subset $\cS \subset \Gamma$ given by $\cS := \{g \in \Gamma \mid (v)_g \neq 0\}$. Since $v \in B \otalg \C \Gamma$, the set $\cS$ is finite. Since $(v)_g = 0$ for all $g \in \cF_0 \Sigma \cF_0^{-1}$, we get that $\cS \cap \cF_0 \Sigma \cF_0^{-1} = \emptyset$. This means that $\cS \cF_0 \Sigma \cap \cF_0 \Sigma = \emptyset$.

Note that $\theta(\abar) = \abar \ot 1$ commutes with $1 \ot P_{\cF_0 \Sigma}$. Hence
$$\theta(\abar) \pi(v) \, (1 \ot P_{\cF_0 \Sigma}) \xi_n = \pi(v) \, (1 \ot P_{\cF_0 \Sigma}) \, \theta(\abar) \, \xi_n$$
lies in the range of $1 \ot P_{\cS \cF_0 \Sigma}$ for all $n$. It then follows from \eqref{eq.thatsit} that
$$\limsup_n \|\xi_n - (1 \ot P_{\cS \cF_0 \Sigma}) \xi_n \| < \sqrt{1-\delta^2} \; .$$
This means that
$$\liminf_n \| (1 \ot P_{\cS \cF_0 \Sigma}) \xi_n \| > \delta \; .$$
We put $\cF_1 := \cS \cF_0 \cup \cF_0$. Since $\cS \cF_0 \Sigma$ is disjoint from $\cF_0 \Sigma$, the vectors $(1 \ot P_{\cS \cF_0 \Sigma}) \xi_n$ and $(1 \ot P_{\cF_0 \Sigma}) \xi_n$ are orthogonal. So in combination with \eqref{eq.bigger}, it follows that \eqref{eq.evenbigger} holds. As explained right after \eqref{eq.evenbigger}, this concludes the proof of Theorem \ref{thm.key} in case~2.

\section{Proofs of Theorems \ref{thm.main-hyp}, \ref{thm.main-two}, \ref{thm.no-Cartan} and \ref{thm.products}}\label{sec.proofs}

\begin{proof}[Proof of Theorem \ref{thm.main-hyp}]
Let $\Gamma = \Gamma_1 \times \cdots \times \Gamma_n$ be a direct product of $n \geq 1$ weakly amenable, nonamenable, bi-exact groups $\Gamma_i$. Take an arbitrary free ergodic pmp action $\Gamma \actson (X,\mu)$. Write $B = \rL^\infty(X)$ and $M = B \rtimes \Gamma$. Let $A \subset M$ be a Cartan subalgebra. Because of \cite[Theorem A.1]{Po01}, it suffices to prove that $A \prec_M B$.

Denote by $\Gammah_i$ the direct product of all $\Gamma_j$ with $j \neq i$. Put $M_i := B \rtimes \Gammah_i$. Fix $i = 1,\ldots,n$ and view $M$ as $M = M_i \rtimes \Gamma_i$. Because $\Gamma_i$ is nonamenable, we have that $M$ is not amenable relative to $M_i$. So, Theorem \ref{thm.key} implies that $A \prec_M M_i$ for all $i = 1,\ldots,n$. Since we know that $M$ is a factor, that $A \subset M$ is regular and that $A \prec_M M_i$ for all $i = 1,\ldots,n$, it follows from \cite[Proposition 2.5 and Lemma 2.6]{Va10} that $A \prec_M B$.

By Lemma \ref{lem.exam-wa-bi-exact}, the groups in 1, 2 and 3 are weakly amenable, nonamenable and bi-exact. So, the theorem applies to these groups and their direct products.
\end{proof}

\begin{proof}[Proof of Theorem \ref{thm.main-two}]
Take $M = B \rtimes \Gamma$ and $A \subset qMq$ as in the formulation of the theorem. Put $P = \cN_{qMq}(A)\dpr$. By Theorem \ref{thm.key}, we have that $P$ is amenable relative to $B$, or that $A \prec_M B$.
\end{proof}

\begin{proof}[Proof of Theorem \ref{thm.no-Cartan}]
Put $M = N \ovt \rL \Gamma$ and assume that $A \subset M$ is a Cartan subalgebra. Since $\Gamma$ is nonamenable, $M$ is not amenable relative to $N$. So Theorem \ref{thm.key} implies that $A \prec_M N$. Taking relative commutants, it follows from \cite[Lemma 3.5]{Va07} that $\rL \Gamma \prec_M M \cap A'$. Since $M \cap A' = A$ and since $\Gamma$ is nonamenable, there are no normal $*$-homomorphisms from $\rL \Gamma$ to an amplification of $A$. Hence the statement $\rL \Gamma \prec_M A$ is absurd.
\end{proof}

\begin{proof}[Proof of Theorem \ref{thm.products}]
Using e.g.\ \cite[Proposition 2.5]{Va10}, we find projections $p_i \in \cZ(P)$ such that $Ap_i \prec^f_M B \rtimes \Gammah_i$ and $A(q-p_i) \not\prec B \rtimes \Gammah_i$ for all $i=1,\ldots,n$. Of course, some or even all of the $p_i$ could be zero. Define $p_0 = q - (p_1 \vee \cdots \vee p_n)$. We consider the subalgebra $Ap_0 \subset p_0 M p_0$ whose normalizer is given by $P p_0$. By Lemma \ref{lem.biexact-products}, the group $\Gamma$ is bi-exact relative to $\cG = \{\Gammah_1,\ldots,\Gammah_n\}$. By construction, $A p_0 \not\prec B \rtimes \Gammah_i$ for all $i=1,\ldots,n$. So, by Theorem \ref{thm.key}, we have that $Pp_0$ is amenable relative to $B$.
\end{proof}

\section{\boldmath Proof of Theorem \ref{thm.main-ME}}

We first prove the following more general result.

\begin{theorem}\label{thm.ME}
Let $\Lambda = \Lambda_1 \times \cdots \times \Lambda_n$ be a direct product of $1 \leq n < \infty$ weakly amenable, bi-exact groups. Assume that $\Gamma$ is a measure equivalence subgroup of $\Lambda$ through an action $\Gamma \times \Lambda \actson \Om$ as in Definition \ref{def.ME}. Then, at least one of the following statements holds.
\begin{enumerate}
\item[(1)] $\Gamma$ is $\cCs$-rigid and $\cC$-rigid.
\item[(2)] There exists an $i \in \{1,\ldots,n\}$, a nonnegligible $(\Gamma \times \Lambda)$-invariant subset $\Om_0 \subset \Om$ and a
sequence of measurable maps $\xi_n : \Om_0 \recht \Prob \Lambda_i$ such that
$$\lim_n \|\xi_n((g,s) \cdot x) - s_i \cdot \xi_n(x)\|_1 = 0 \quad\text{for all}\;\; g \in \Gamma, s \in \Lambda \;\;\text{and a.e.}\;\; x \in \Om_0 \; .$$
\end{enumerate}
Moreover in the following two special cases, statement (2) has a simpler equivalent formulation.
\begin{itemize}
\item[(a)] If also the fundamental domain of $\Gamma \actson \Om$ has finite measure (i.e.\ $\Gamma$ and $\Lambda$ are measure equivalent), then statement (2) is equivalent with the amenability of one of the $\Lambda_i$.
\item[(b)] Assume that the $\Lambda_i < G_i$ are lattices in the second countable locally compact groups $G_i$ and assume that $\Gamma < G = G_1 \times \cdots \times G_n$ is a countable closed subgroup. If we take $\Om = G$ with the action of $\Gamma \times \Lambda \actson \Om$ given by left right multiplication, then statement (2) is equivalent with the image of $\Gamma$ in one of the $G_i$ having an amenable closure.
\end{itemize}
\end{theorem}

Note that Theorem \ref{thm.main-ME} is a direct consequence of Theorem \ref{thm.ME}. Case~1 of Theorem \ref{thm.main-ME} follows immediately by using (a) in Theorem \ref{thm.ME}. To prove case~2 of Theorem \ref{thm.main-ME}, we choose cocompact lattices $\Lambda_i < G_i$. Then all $\Lambda_i$ are non-elementary hyperbolic groups. So, by Lemma \ref{lem.exam-wa-bi-exact}, all $\Lambda_i$ are weakly amenable, nonamenable and bi-exact. So case~2 of Theorem \ref{thm.main-ME} follows by using (b) in Theorem \ref{thm.ME}.

\begin{proof}[Proof of Theorem \ref{thm.ME}]
Choose a free ergodic pmp action $\Gamma \actson X$. Put $M = \rL^\infty(X) \rtimes \Gamma$ and assume that $A \subset M$ is a maximal abelian von Neumann subalgebra whose normalizer $\cN_M(A)\dpr$ is a finite index subfactor of $M$. From \cite[Theorem A.1]{Po01}, we know that $A$ is unitarily conjugate with $\rL^\infty(X)$ if and only if $A \prec_M \rL^\infty(X)$. We will prove that either $A \prec_M \rL^\infty(X)$ or that (2) holds.

We denote the commuting actions $\Gamma \actson \Om$ and $\Lambda \actson \Om$ as actions on the left, resp.\ on the right, and denote these actions by a dot $\cdot$. Choose a fundamental domain $\cU \subset \Om$ for the action $\Gamma \actson \Om$ and choose a fundamental domain $\cV \subset \Om$ for the action $\Lambda \actson \Om$. So, up to measure zero, we get partitions
$$\Om = \bigsqcup_{g \in \Gamma} g \cdot \cV \quad\text{and}\quad G = \bigsqcup_{s \in \Lambda} \cU \cdot s \; .$$
Since $\cU$ is of finite measure, we may normalize $m$ such that $m(\cU) = 1$.

We identify $\Om/\Lambda = \cU$. Through this identification, the natural action $\Gamma \actson \Om/\Lambda$ becomes a pmp action $\Gamma \actson \cU$ that we denote by $\ast$ to distinguish it from
the action $\Gamma \actson \Om$ denoted by $\cdot$. We then get the $1$-cocycle $\om : \Gamma \times \cU \recht \Lambda$ for the action $\Gamma \overset{\ast}{\actson} \cU$ such that
$$g \cdot x = (g \ast x) \cdot \om(g,x) \quad\text{for all}\;\; g \in \Gamma \;\;\text{and a.e.}\;\; x \in \cU \; .$$
In particular, for $g \in \Gamma$ and a.e.\ $x \in \cU$, we have $\om(g,x) = s$ if and only if $g \cdot x \in \cU \cdot s$.

Define the tracial von Neumann algebra $N := \rL^\infty(X \times \cU) \rtimes \Gamma$, where $\Gamma$ acts on $X \times \cU$ diagonally. We view $M$ as a von Neumann subalgebra of $N$ in the canonical way. For every $g \in \Gamma$, we denote by $V_g \in \rL^\infty(\cU) \ovt \rL \Lambda$ the unitary given by $V_g(x) = v_{\om(g,x)}$. Here we use the notation $(v_s)_{s \in \Lambda}$ to denote the canonical unitaries in $\rL \Lambda$. We then get a normal trace preserving $*$-homomorphism
$$\Delta : N \recht N \ovt \rL \Lambda : \Delta(a u_g) = (a u_g \ot 1) V_g \quad\text{for all}\;\; a \in \rL^\infty(X \times \cU), g \in \Gamma \; .$$
We put $N_i := N \ovt \rL \Lambdah_i$ and identify $N \ovt \rL \Lambda = N_i \ovt \rL \Lambda_i$. As such we view $N \ovt \rL \Lambda$ as the crossed product of $N_i$ and $\Lambda_i$ w.r.t.\ the trivial action of $\Lambda_i$ on $N_i$. Since $\Lambda_i$ is weakly amenable and bi-exact, we will apply Theorem \ref{thm.key} to this crossed product.

Denoting $P = \cN_M(A)\dpr$, we distinguish two cases.

{\bf Case 1.} For all $i \in \{1,\ldots,n\}$ and all nonzero projections $p \in \Delta(P)' \cap (N \ovt \rL \Lambda)$, we have that $\Delta(P)p$ is not amenable relative to $N_i$.

{\bf Case 2.} There exists an $i \in \{1,\ldots,n\}$ and a nonzero projection $p \in \Delta(P)' \cap (N \ovt \rL \Lambda)$ such that $\Delta(P) p$ is amenable relative to $N_i$.

We prove that in case 1, we have $A \prec_M \rL^\infty(X)$, while in case 2, there exists a sequence $\xi_n$ satisfying the conditions of statement (2).

{\bf Proof of case 1.} It follows from Theorem \ref{thm.key} that $\Delta(A)p \prec N_i$ for all $i \in \{1,\ldots,n\}$ and all nonzero projections $p \in \Delta(P)' \cap (N \ovt \rL \Lambda)$. So, using e.g.\ \cite[Proposition 2.6]{Va10}, we get that $\Delta(A) \prec^f N \ovt \rL(\Lambdah_i)$ for all $i \in \{1,\ldots,n\}$. But then $\Delta(A) \prec^f N \ot 1$ (see e.g.\ \cite[Lemma 2.7]{Va10}).

For every subset $\cF \subset \Lambda$, we denote by $P_\cF$ the orthogonal projection of $\ell^2(\Lambda)$ onto $\ell^2(\cF)$. We similarly denote, for every subset $\cF' \subset \Gamma$, by $P_{\cF'}$ the orthogonal projection of $\rL^2(M)$ onto the closed linear span of $\{b u_g \mid b \in \rL^\infty(X), g \in \cF'\}$. Choose $\eps > 0$. We now prove that there exists a finite subset $\cF' \subset \Gamma$ such that
\begin{equation}\label{eq.ouraim}
\|P_{\cF'}(a)\|_2^2 > 1-2\eps \quad\text{for all}\;\; a \in \cU(A) \; .
\end{equation}
Once \eqref{eq.ouraim} is proven, the required conclusion $A \prec_M \rL^\infty(X)$ follows from Definition \ref{def.intertwine}.

Since $\Delta(A) \prec^f N \ot 1$, we get from \cite[Lemma 2.5]{Va10}, a finite subset $\cF \subset \Lambda$ such that
\begin{equation}\label{eq.anapproximation}
\|(1 \ot P_{\cF}) \Delta(a) \|_2^2 > 1-\eps \quad\text{for all}\;\; a \in \cU(A) \; .
\end{equation}
For every $a \in M$, we denote by $a = \sum_{g \in \Gamma} (a)_g u_g$, with $(a)_g \in \rL^\infty(X)$, the Fourier decomposition of $a$. A direct computation yields
\begin{equation}\label{eq.hierennu}
\|(1 \ot P_{\cF})\Delta(a)\|_2^2 = \sum_{g \in \Gamma} \|(a)_g\|_2^2 \;\; m\bigl(\{x \in \cU \mid \om(g,x) \in \cF \}\bigr) \; .
\end{equation}
Note that $\om(g,x) \in \cF$ if and only if $g \cdot x \in \cU \cdot \cF$.
Since $\cU \subset \Om$ has finite measure and $\Gamma \actson \Om$ admits a fundamental domain, there exists a finite subset $\cF' \subset \Gamma$ such that
$$m\bigl(\{x \in \cU \mid \om(g,x) \in \cF \}\bigr) < \eps \quad\text{for all}\;\; g \in \Gamma - \cF' \; .$$
A combination of \eqref{eq.hierennu} and \eqref{eq.anapproximation} then yields \eqref{eq.ouraim}. This concludes the proof of case~1.

{\bf Proof of case 2.} Since $P$ is a finite index subfactor of $M$, Lemma \ref{lem.rel-amen} provides a projection $q \geq p$ such that $q \in \Delta(M)' \cap (N \ovt \rL \Lambda)$ and such that $\Delta(M) q$ is amenable relative to $N_i$.
Write $\cN := \langle N \ovt \rL \Lambda, e_{N_i} \rangle$. We get a $\Delta(M)q$-central positive functional $\Psi_1$ on $q \cN q$ such that $\Psi_1(x) = \tau(x)$ for all $x \in q (N \ovt \rL \Lambda) q$. We identify $\cN = N_i \ovt \B(\ell^2(\Lambda_i))$. As such, we view $\rL^\infty(\cU \times \Lambda_i) = \rL^\infty(\cU) \ovt \ell^\infty(\Lambda_i)$ as a von Neumann subalgebra of $\cN$. The unitaries $\Delta(u_g) \in \cN$, $g \in \Gamma$, normalize $\rL^\infty(\cU \times \Lambda_i) \subset \cN$ and induce the action $\Gamma \actson \cU \times \Lambda_i$ given by $g \cdot (x,s) = (g * x, \om(g,x)_i \, s)$. The formula $\Psi(F) = \Psi_1(q F q)$ then provides a nonzero positive $\Gamma$-invariant functional on $\rL^\infty(\cU \times \Lambda_i)$ such that the restriction of $\Psi$ to $\rL^\infty(\cU)$ is normal and $\Gamma$-invariant.

Denote by $\cW \subset \cU$ the support of $\Psi_{|\rL^\infty(\cU)}$. Then, $\cW$ is a nonnegligible $\Gamma$-invariant subset of $\cU$. Modifying $\Psi$ by using the $\Gamma$-invariant Radon-Nikodym derivative between $\Psi_{|\rL^\infty(\cW)}$ and integration w.r.t.\ $m$, we may assume that $\Psi_{|\rL^\infty(\cW)}$ equals integration w.r.t.\ $m$. Approximating $\Psi \in \rL^\infty(\cW \times \Lambda_i)^*$ by a net of elements in $\rL^1(\cW \times \Lambda_i)^+$ and passing to convex combinations, we can find a sequence of measurable maps $\psi_n : \cW \recht \Prob \Lambda_i$ such that
$$\lim_n \|\psi_n(g * x) - \om(g,x)_i \cdot \psi_n(x)\|_1 = 0 \quad\text{for all}\;\; g \in \Gamma \;\;\text{and a.e.}\;\; x \in \cW \; .$$
Define $\Om_0 := \cW \cdot \Lambda$. Then, $\Om_0$ is a nonnegligible $(\Gamma \times \Lambda)$-invariant subset of $\Om$. Defining $\xi_n : \Om_0 \recht \Prob \Lambda_i$ given by $\xi_n(x \cdot s) := s_i^{-1} \cdot \psi_n(x)$ for all $x \in \cW$ and $s \in \Lambda$, it is easy to check that
$$\lim_n \|\xi_n(g \cdot y \cdot r^{-1}) - r_i \cdot \xi_n(y)\|_1 = 0 \quad\text{for all}\;\; g \in \Gamma, r \in \Lambda \;\;\text{and a.e.}\;\; y \in \Om_0 \; .$$
So, we have shown that (2) holds. This concludes the proof of case~2.

It remains to prove (a) and (b).

(a)\ If $\Lambda_i$ is amenable, take a sequence $\eta_n \in \Prob \Lambda_i$ such that $\lim_n \|s \cdot \eta_n - \eta_n\|_1 = 0$ for all $s \in \Lambda_i$. Defining $\xi_n(x) = \eta_n$ for all $x \in \Om$, it follows that (2) holds. Conversely, assume that the sequence $\xi_n : \Om_0 \recht \Prob \Lambda_i$ satisfies (2) and that also $m(\cV) < \infty$. Identifying $\Gamma \backslash \Om$ with $\cV$, we get a right action of $\Lambda$ on $\cV$, denoted by $*$, and a $1$-cocycle $\mu : \cV \times \Lambda \recht \Gamma$ such that
$$y \cdot s = \mu(y,s) \cdot (y * s) \quad\text{for a.e.\ $y \in \cV$ and all $s \in \Lambda$.}$$
Consider the action $\Lambda \actson \cV \times \Lambda_i$ given by $s \cdot (y,t) = (y * s^{-1}, s_i t)$. Since $\Om_0 \subset \Om$ is nonnegligible and $(\Gamma \times \Lambda)$-invariant, also $\cV \cap \Om_0$ is nonnegligible and invariant under the $*$-action of $\Lambda$. The restriction of $\xi_n$ to $\cV \cap \Om_0$ defines a sequence $\zeta_n \in \rL^1((\cV \cap \Om_0) \times \Lambda_i)^+$ with $\|\zeta_n\|_1 = m(\cV \cap \Om_0)$ for all $n$, that satisfies $\lim_n \|s \cdot \zeta_n - \zeta_n\|_1 = 0$ for all $s \in \Lambda$, by the dominated convergence theorem. The push forward of $\zeta_n$ along the factor map $\cV \times \Lambda_i \recht \Lambda_i : (y,r) \mapsto r$, then provides, after rescaling, a sequence $\eta_n \in \Prob \Lambda_i$ such that $\lim_n \|s \cdot \eta_n - \eta_n\|_1 = 0$ for all $s \in \Lambda_i$. So, $\Lambda_i$ is amenable.

(b)\ Assume that $\Om = G = G_1 \times \cdots \times G_n$ as in statement (b). When $x \in G$, we denote by $x_i$ the $i$'th component of $g$. First assume that we have a sequence $\xi_n : \Om_0 \recht \Prob \Lambda_i$ as in (2). We consider the action $\Gamma \actson \cU \times \Lambda_i$ given by $g \cdot (x,s) = (g * x , \om(g,x)_i \, s)$. As in the proof of (a), $\cU \cap \Om_0$ is nonnegligible, invariant under the $*$-action of $\Gamma$ on $\cU$, and the restriction of $\xi_n$ to $\cU \cap \Om_0$ defines a sequence $\zeta_n \in \rL^1((\cU \cap \Om_0) \times \Lambda_i)^+$ with $\|\zeta_n\|_1 = m(\cU \cap \Om_0)$ for all $n$, that satisfies $\lim_n \|g \cdot \zeta_n - \zeta_n\|_1 = 0$ for all $g \in \Gamma$. The push forward of $\zeta_n$ along the factor map $\cU \times \Lambda_i \recht G_i : (x,s) \mapsto x_i s$, then provides, after rescaling, a sequence of unit vectors $\eta_n \in \rL^1(G_i)^+$ such that $\lim_n \|g_i \cdot \eta_n - \eta_n\|_1 = 0$ for all $g \in \Gamma$. Denoting $\pi_i : G \recht G_i : \pi_i(g) = g_i$, it follows from Lemma \ref{lem.criterion-amen} that $\pi_i(\Gamma)$ has an amenable closure inside $G_i$.

Conversely, assume that the closure $H$ of $\pi_i(\Gamma)$ in $G_i$ is an amenable locally compact group. Then the action $\Lambda_i \actson H \backslash G_i$ is amenable in the sense of Zimmer. So, we find a sequence of measurable maps $\psi_n : H \backslash G_i \recht \Prob \Lambda_i$ such that $\lim_n \|\psi_n(x \cdot s) - s^{-1} \cdot \psi_n(x)\|_1 = 0$ for a.e.\ $x \in H \backslash G_i$ and all $s \in \Lambda_i$. We view $\psi_n$ as a sequence of $H$-invariant measurable maps $\psi_n : G_i \recht \Prob \Lambda_i$. Defining $\xi_n : G \recht \Prob \Lambda_i : \xi_n(x) = \psi_n(x_i)$, statement (2) is satisfied.
\end{proof}

For the convenience of the reader, we include a proof of the following standard, but slightly technical, lemma.

\begin{lemma}\label{lem.criterion-amen}
Let $G$ be a second countable locally compact group and $H_0 < G$ any subgroup. Assume that $\rL^\infty(G)$ admits a state that is invariant under left translation by all elements of $H_0$. Then the closure of $H_0$ is an amenable locally compact group.
\end{lemma}
\begin{proof}
Let $\om$ be a left $H_0$-invariant state on $\rL^\infty(G)$. Denote by $H$ the closure of $H_0$ inside $G$. Since $G$ is second countable, we can take a Borel map $T : G \recht H$ satisfying $T(hg) = h T(g)$ for all $h \in H$, $g \in G$ (see e.g.\ \cite[Theorem 12.17]{Ke95}). Consider the Banach space $\rLUC(H)$ of bounded left uniformly continuous functions on $H$. Then, $\mu : f \mapsto \om(f \circ T)$ is a state on $\rLUC(H)$ that is invariant under left translation by $H_0$. By continuity, $\mu$ is invariant under left translation by $H$. So, $\rLUC(H)$ admits a left invariant mean. This implies that $H$ is amenable (see e.g.\ \cite[Theorem G.3.1]{BHV08}).
\end{proof}

\section{\boldmath An application to W$^*$-rigidity} \label{sec.rigidity}

Recall that an ergodic pmp action $\Gamma \actson (X,\mu)$ is called \emph{aperiodic} if all finite index subgroups of $\Gamma$ still act ergodically. Following \cite[Definition 1.8]{MS02} an ergodic pmp action $\Lambda \actson (Y,\eta)$ is called \emph{mildly mixing} if there are no nontrivial recurrent subsets: if $A \subset Y$ is measurable and $\liminf_{g \recht \infty} \eta(g \cdot A \, \triangle \, A) = 0$, then $\eta(A) = 0$ or $\eta(A)=1$. Note that for a mildly mixing action $\Lambda \actson (Y,\eta)$ all infinite subgroups of $\Lambda$ act ergodically on $(Y,\eta)$. Finally, a pmp action $\Gamma_1 \times \Gamma_2 \actson (X,\mu)$ of a product group is called \emph{irreducible} if both $\Gamma_1$ and $\Gamma_2$ act ergodically.

\begin{proof}[Proof of Theorem \ref{thm.MoSh}]
Since $\Gamma$ is a product of hyperbolic groups, Theorem \ref{thm.main-hyp} applies. So the existence of an isomorphism $\rL^\infty(X) \rtimes \Gamma \cong \rL^\infty(Y) \rtimes \Lambda$ implies that $\Gamma \actson (X,\mu)$ and $\Lambda \actson (Y,\eta)$ are orbit equivalent. Since non-elementary hyperbolic groups belong to the class $\cC_{\text{reg}}$ of Monod and Shalom, it follows from \cite[Theorem 1.10]{MS02} that the groups $\Gamma$ and $\Lambda$ must be isomorphic and that their actions must be conjugate.
\end{proof}

\end{document}